\documentclass[a4paper,11pt]{article}

\usepackage[utf8]{inputenc}    
\usepackage[T1]{fontenc}        
\usepackage{textcomp}  
\usepackage[french, english]{babel}
\usepackage{mathrsfs} 
\usepackage{amssymb}
\usepackage{amsmath}
\usepackage{variations}
\usepackage{amsthm}
\usepackage{fullpage}
\usepackage[retainorgcmds]{IEEEtrantools}

\newtheorem{Thm}{Theorem}
\newtheorem{Prop}{Proposition}
\newtheorem{Lem}{Lemma}
\newtheorem{Rq}{Remark}

\newtheorem{Def}{Definition}

\newcommand{\N}{N_{\rho}}

\newcommand{\Mr}{\widetilde{M}_{\rho}}

\newcommand{\I}{\int_{\mathbb{R}^d}\int_{\mathbb{R}_+}}
\newcommand{\ma}{\mathcal{M}_{\alpha ,\beta_1, \beta_2}}
\newcommand{\maa}{\mathcal{M}_{\alpha' ,\beta_1, \beta_2}}
\newcommand{\mb}{\mathcal{M}_{\alpha ,\beta}}

\newcommand{\mcal}{\mathcal{M}}

\newcommand{\NN}{\widetilde{N}_{\rho}}

\newcommand{\m}{M_{\rho}}

\newcommand{\II}{\int_{\mathbb{R}^d}\int_{\mathbb{R}_+}\int_{\mathbb{R}}}
\newcommand{\Rd}{\mathbb{R}^d}
\newcommand{\Z}{\mathcal{Z}(\mathbb{R}^d)}

\title{Random Balls Model With Dependence}

\author{Renan \bsc{GOBARD} \thanks{IRMAR, Université de Rennes 1, 35042 Rennes. Email: renan.gobard@univ-rennes1.fr}}

\begin{document}

\maketitle

\begin{abstract}
In this article, we consider a configuration of weighted random balls in $\Rd$ generated according to a Poisson point process. The model investigated exhibits inhomogeneity, as well as dependence between the centers and the radii and heavy tails phenomena. We investigate the asymptotic behaviour of the total mass of the configuration of the balls at a macroscopic level. Three different regimes appear depending on the intensity parameters and the zooming factor. Among the three limiting fields, two are stable while the third one is a Poisson integral bridging between the two stable regimes. For some particular choices of the inhomogeneity function, the limiting fields exhibit isotropy or self-similarity. 
\end{abstract}

\section{Introduction}

In this work, we generalize the weighted random balls model introduced in \cite{KT} and \cite{Jcb} into an inhomogeneous weighted random balls model including dependence. This model is an aggregation of weighted random balls whose centers and radii are distributed, as a couple $(x,r)$, according to a Poisson point process. The weight of a ball is independent of its center and radius. On this model we consider the field that, at each point, is the sum of the weights of the balls containing the point. We investigate macroscopic behaviour of such a quantity, that is the convergence of the model while performing a zoom-out. As it is common in such models (see \cite{Bek}, \cite{Jcb}, \cite{Kaj} and \cite{KT}), suitable scalings and normalizations allow us to exhibit three different limit fields: a stable field with dependence, a Poissonian field and a stable field with independence. 

The differences between this model and the weighted random balls model in \cite{Jcb} rely on two points:
\begin{itemize}
\item the dependence between the radii and the centers that lies in the tail index of the radii,
\item the inhomogeneity that comes first from the dependence but that is also reinforced by the introduction of an inhomogeneous intensity for the Poisson point process.
\end{itemize} 
By relaxing the dependence with a constant radius tail index, we recover as a particular model one with inhomogeneous random balls whose radius is independent of the center. We distribute the weights in the same way as in \cite{Jcb}, that is in the domain of attraction of an $\alpha$-stable distribution. 

We briefly discuss the models investigated so far. In \cite{MR} and \cite{KT}, weighted random balls models are studied in dimension 1. In \cite{Be}, a random balls model (with weight constant equal to 1) is studied under a zoom-in scaling, whereas it is investigated under a zoom-out scaling in \cite{Kaj}. Both approaches are gathered in \cite{Bek}, which gives a general framework to allow both zoom-in and zoom-out procedures. This framework is then used in \cite{Jcb}, to study the so-called weighted random balls model.
The process arising under the intermediate regime in \cite{KT} (which correspond to the Poissonian field in our study), is deeply studied in \cite{Gai} and seen as a bridge between the two other processes obtained (the two stable fields in our case). 

The main contribution of this paper is the introduction of dependence and inhomogeneity in such models. In dimension $1$, where those model recover queuing problems like packet networks computer traffic (see \cite{KT}), it models rates of connection and lengths of connection that depend of the time. Random balls models also apply in dimension $2$ to imagery (see \cite{Be}) or to wireless network (see \cite{Kaj}). In that latter example, the inhomogeneity takes into account that the number of antennas as well as the radius of emission is different from a region to another. In higher dimension, it models porous or heterogeneous media (see \cite{Be} again).

 We describe the setting in Section \ref{sec:mod}. Our main results are stated, discussed and proved in Section \ref{sec:thm}. In the last section, we briefly discuss about the asymptotic behaviour of the model at a microscopic level.

\section{The inhomogeneous weighted random balls model with dependence}\label{sec:mod}

We are interested in a weighted random balls model where the triplets $(x,r,m)$ (i.e. center, radius and weight) are generated according to a Poisson random measure $N(dx,dr,dm)$ with intensity $f(x,r)dxdrG(dm)$, where $f$ is a positive function defined on $\mathbb{R}^d\times \mathbb{R}^+$ and $G$ is a probability measure on $\mathbb{R}$.

The absolutely continuous measure $f(x,r)dxdr=F_x(dr)dx$ drives the distribution of the centers and the radii. All along, we make the assumption that 
\begin{align}
&r\mapsto \Vert f(\cdot,r) \Vert_{\infty} \mbox{ is continuous},  \label{cf1} \\
&\int_{\mathbb{R}^+} r^{d} \Vert f(\cdot,r) \Vert_{\infty} dr<+\infty. \label{cf2}
\end{align}
We need also to know the behaviour of $f$ for large radii. In that purpose, we suppose that :
\begin{equation}
f(x,r)\sim_{r\to +\infty} \frac{g(x)}{r^{\beta(x)+1}}, \label{cf}\\
\end{equation}
uniformly in $x\in \Rd$, where $g$ and $\beta$ are positive functions on $\Rd$. 

Roughly speaking, the dependence between the centers and the radii lies in the function $\beta$ whereas the inhomogeneity mostly lies in the function $g$. Note that Condition \eqref{cf1} and \eqref{cf} imply $g\in L^{\infty}(\Rd)$.We suppose that there exist two constants $\beta_1$ and $\beta_2$ such that
\begin{equation}
d<\beta_1 \leq \beta(x) \leq \beta_2.\label{cbeta}
\end{equation}
Typically, taking $\beta=\beta_11_{B_1}+\beta_{2}1_{B_2}$, with $B_1\sqcup B_2=\Rd$, means that the radii have two typical different behaviours on $B_1$ and $B_2$. General choices of functions $\beta$ satisfying \eqref{cbeta} give continuous mixtures of these behaviours. Choosing $\beta$ as a constant function gives us an inhomogeneous model  with (asymptotically) no dependence.  Note that Conditions \eqref{cf2} and \eqref{cf} require $d<\beta_1$. Moreover the parameter $\beta_2$ is a technical parameter and do not appear in the final results.

 We suppose that the probability measure $G$ belongs to the domain of attraction of the $\alpha$-stable distribution $S_{\alpha}(\sigma, b, \tau)$ (using the terminology from \cite{ST}) with $\alpha \in \left(1, 2\right]$. Taking $\alpha=2$ implies that $G$ is any finite variance distribution and recovers the typical weight of \cite{KT}. When $\alpha\in (1,2)$, a typical choice could be a heavy tailed distribution.

\bigskip

\begin{Rq}
The choice $f(x,r)=f(r)$ (and thus $g$=1 and $\beta$ is a constant function) recovers the setting of \cite{Jcb}.
\end{Rq}

\bigskip

Let $y\in \Rd$, the algebraic weights of all the balls containing $y$ is given by 
\begin{displaymath}
\sum M_i\delta_y(B(X_i,R_i))=\II m\delta_y(B(x,r))N(dx,dr,dm)=M(\delta_y),
\end{displaymath}
where $(X_i,R_i,M_i)$ are generated according to the Poisson measure $N$. To study the mass generated by the model, we generalize the random field $M$ to all signed measures $\mu$ on $\mathbb{R}^d$ with a finite total variation (whose set is denoted by $\mathcal{Z}(\mathbb{R}^d)$):
\begin{displaymath}
\forall \mu \in \mathcal{Z}(\mathbb{R}^d), \: M(\mu)= \II m\mu(B(x,r))N(dx,dr,dm).
\end{displaymath}

The above stochastic integral is well defined and has a finite expectation, according to the following proposition. 
\begin{Prop} For all $\mu \in \mathcal{Z}(\Rd)$, we have $\mathbb{E}\left[  \vert M(\mu)\vert \right]<+\infty$ and
\begin{equation}\label{prop:exp}
\mathbb{E}\left[ M(\mu) \right]=\II m\mu(Bx,r)f(x,r)dxdrG(dm).	
\end{equation}
\end{Prop}

\begin{proof}

We aim at proving that for all $\mu \in \mathcal{Z}(\mathbb{R}^d)$:
\begin{equation} \label{Cex}
\begin{split}
\mathbb{E}\left[\vert M(\mu)\vert \right] &=\II \vert m\mu(B(x,r))\vert f(x,r)dxdrG(dm) \\
&=\int_{\mathbb{R}}\vert m\vert G(dm)\I \vert \mu(B(x,r))\vert f(x,r)dxdr< +\infty .
\end{split}
\end{equation}

\bigskip

Since $G$ has a finite expectation due to $\alpha >1$, we focus on the second integral. Using Conditions \eqref{cf1} and \eqref{cf2} we have with Fubini Theorem
\begin{displaymath}
\label{cdef}
\begin{split}
\I \vert \mu(B(x,r))\vert  f(x,r)dxdr &\leq \int_{\mathbb{R}^d} \left( \int_{\mathbb{R}^+} \int_{\mathbb{R}^d}  1_{B(y,r)}(x)  f(x,r)dxdr \right) \vert \mu \vert (dy)\\
&\leq \int_{\mathbb{R}^d} \left( \int_{\mathbb{R}^+} \int_{\mathbb{R}^d}  1_{B(y,r)}(x)dx \Vert f(\cdot, r)\Vert_{\infty}dr \right) \vert \mu \vert (dy)\\
& \leq \int_{\mathbb{R}^d} c_d \int_{\mathbb{R}^+} r^d\Vert f(\cdot, r)\Vert_{\infty}dr\vert \mu \vert (dy)\\
&\leq \vert\mu\vert(\Rd) c_d \int_{\mathbb{R}^+} r^d\Vert f(\cdot, r)\Vert_{\infty}dr<+\infty,
\end{split}
\end{displaymath}
where $c_d$ is the volume of the unit Euclidean ball of $\Rd$. This proves \eqref{Cex} and ensures $M(\mu)$ has the finite expectation given in \eqref{prop:exp}.

\end{proof}

\section{Asymptotics}\label{sec:thm}

We analyse the model at a macroscopic level by performing a zoom-out. Heuristically, we introduce a parameter $\rho$, representing the rate of zoom-out: it reduces the mean radius. Accordingly we need to increase the number of centers to obtain something significant at the limit. Indeed, if we just make the change of variable $r\mapsto\rho r$ in the density, using \eqref{cf} we would have
\begin{equation}\label{cfr0}
f(x,r/\rho)\sim_{r \to +\infty} \frac{g(x)}{r^{\beta(x)+1}}\rho^{\beta(x)+1} \to_{\rho \to 0} 0,
\end{equation}
and asymptotics would just vanish. As a consequence, in order to compensate this convergence to $0$, we use a scaled version $f_\rho$ of $f$ such that:
\begin{equation}\label{cfr00}
f_\rho(x,r/\rho)\sim_{r\to +\infty } \frac{g_\rho(x)}{r^{\beta(x)+1}}\rho^{\beta(x)+1}
\end{equation}
and we suppose that $g_\rho(x)\sim_{\rho \to 0} \lambda(\rho) g(x)$ and $\lambda(\rho)\to +\infty$ when $\rho \to 0$. Of course we cannot chain these two latter equivalences so instead we replace \eqref{cf} and \eqref{cfr00} by
\begin{equation}
f_\rho(x,r)\sim_{r\to +\infty} \lambda(\rho) \frac{g(x)}{r^{\beta(x)+1}}, \label{cfr}
\end{equation} 
uniformly both in $x$ and $\rho$. For technical purpose, we also suppose that, for any $\rho>0$, 
 \begin{equation}
 \forall r \in \mathbb{R}^+, \: \Vert f_\rho(\cdot, r)\Vert_\infty\leq \lambda(\rho)\Vert f(\cdot, r)\Vert_\infty.  \label{cfr1} \\
 \end{equation} 
 The parameter $\lambda(\rho)$ plays a role in the density of centers in the sense that:
 \begin{displaymath}
 \lim_{r\to \infty} \frac{1}{c_dr^d}\int_{B(0,r)}g_{\rho}(x)dx\sim \lambda(\rho)\lim_{r\to \infty} \frac{1}{c_dr^d}\int_{B(0,r)}g(x)dx.
 \end{displaymath}
 Thus, when $\lambda(\rho)$ increases, the mean number of balls increases accordingly which compensates the decreasing of the volume of the balls. 
\begin{Rq}
With these notation, the settings of \cite{Jcb} is recovered with $f_{\rho}(x,r)=\lambda(\rho)f(r)$.
\end{Rq}
Our study focuses on the weight random functional given by 
\begin{displaymath}
\m(\mu)=\II m\mu(B(x,r)) \N(dx,dr,dm)
\end{displaymath}
when $\rho\to 0$, where $\N$ stands for the Poisson random measure with the intensity obtained by the scaling $r\to \rho r$
\begin{equation}\label{int}
\rho^{-1}f_\rho(x,r/\rho) dxdr G(dm).
\end{equation} 
 To investigate the fluctuations of the process, we need to center it and exhibit a normalization factor $n(\rho)$: 
\begin{displaymath}
n(\rho)^{-1}\Mr(\mu)=\frac{\m(\mu)-\mathbb{E}\left[\m(\mu)\right]}{n(\rho)}=\II n(\rho)^{-1} m\mu(B(x,r)) \NN(dx,dr,dm),
\end{displaymath}
where $\NN$ stands for the compensated Poisson random measure with intensity \eqref{int}. The characteristic function is well known and its exponent is given by
\begin{equation}\label{fcm}
\begin{split}
\log\left(\varphi_{n(\rho)^{-1}\Mr(\mu)}(\theta)\right)&=\log\left(\mathbb{E}\left[e^{i\theta n(\rho)^{-1}\Mr(\mu)}\right]\right)\\
&=\II \psi\left(\theta n(\rho)^{-1}m\mu(B(x,r))\right)\rho^{-1}f_\rho(x,r/\rho) dxdr G(dm)\\
&=\I \psi_G\left(\theta n(\rho)^{-1}\mu(B(x,r))\right)\rho^{-1}f_\rho(x,r/\rho) dxdr ,
\end{split}
\end{equation}
where $\psi_G(u)=\int_{\mathbb{R}}\psi(mu)G(dm)$ and $\psi(u)=e^{iu}-1-iu$. Note that the function $\psi_G$ is  Lipschitzian with constant $2\int_{\mathbb{R}}\vert m\vert G(dm)$. Indeed, for $u,v\in \mathbb{R}$, we have:
\begin{displaymath}
\vert \psi_G(u)-\psi_G(v)\vert\leq \int_{\mathbb{R}}\vert \psi(mu)-\psi(mv)\vert G(dm)\leq 2\left(\int_{\mathbb{R}}\vert m\vert G(dm)\right) \vert u-v\vert.
\end{displaymath}
 
\bigskip

The asymptotic results depend on three scaling regimes that one can understand by looking at the mean of the number of balls that cover the origin and with a radius large enough. Heuristically, we have from \eqref{cfr}
\begin{displaymath}
\begin{split}
\mathbb{E}\left[\#\left\lbrace (x,r): 0\in B(x,r), r>1\right\rbrace\right] &=\int \int \int_{\left\lbrace (x,r): 0\in B(x,r), r>1\right\rbrace} \rho^{-1}f_\rho(x,r/\rho)dxdrG(dm)\\
&=\int \int_{\left\lbrace (x,r): 0\in B(x,r), r>1\right\rbrace} \rho^{-1}f_\rho(x,r/\rho)dxdr \\
&\sim_{\rho\to 0} \lambda(\rho) \int \int_{\left\lbrace (x,r) : 0\in B(x,r), r>1 \right\rbrace} \frac{\rho^{\beta(x)}g(x)}{r^{\beta(x)+1}} dxdr.
\end{split}
\end{displaymath}

Set $B_1=\left\lbrace x\in \Rd: \beta(x)=\beta_1 \right\rbrace$. Under condition \eqref{cbeta}, for $\rho$ small enough:
\begin{multline*}
\lambda(\rho) \rho^{\beta_1} \int \int_{\left\lbrace (x,r) : 0\in B(x,r), r>1 \right\rbrace} \frac{g(x)1_{B_1}(x)}{r^{\beta_1)+1}} dxdr \\ \leq \lambda(\rho) \int \int_{\left\lbrace (x,r) : 0\in B(x,r), r>1 \right\rbrace} \frac{\rho^{\beta(x)}g(x)}{r^{\beta(x)+1}} dxdr \\ \leq \lambda(\rho) \rho^{\beta_1} \int \int_{\left\lbrace (x,r) : 0\in B(x,r), r>1 \right\rbrace} \frac{g(x)}{r^{\beta(x)+1}} dxdr.
\end{multline*}
Since $\int \int_{\left\lbrace (x,r) : 0\in B(x,r), r>1 \right\rbrace} \frac{g(x)}{r^{\beta(x)+1}} dxdr<+\infty$, if the set $B_1$ is non negligible,  three regimes appear:
\begin{itemize}
\item small-balls scaling: $\lambda(\rho)\rho^{\beta_1}\to 0$,
\item intermediate scaling: $\lambda(\rho)\rho^{\beta_1} \to \ell \in (0,+\infty)$,
\item large-balls scaling: $\lambda(\rho)\rho^{\beta_1 } \to + \infty$.
\end{itemize}
Note that the scaling regimes do not depend on the value $\beta_2$ or even other intermediate value of $\beta$ but only on $\beta_1$. This can be understood as follows, when we reduce the mean radius by $\rho$ , the normalization will only compensate the decreasing of the biggest radii, that is the radii with the smallest tail-index $\beta_1$, and the contribution of the other balss wille be negligible.

\bigskip

To ensure asymptotics make sense, it is necessary to consider configuration $\mu$ for which the limit processes are well defined. To that purpose, we introduce the following subset of $\Z$.
\begin{Def}
Let $1<\alpha\leq 2$ and $0<\beta_1\leq \beta_2$. The set $\ma$ ($\mcal$ when there is no ambiguity) consists of (signed) measures $\mu \in \Z$ such that there exist two real numbers $s$ and $t$ with $0<s<\beta_1 \leq \beta_2< t$ and a positive constant $C$ such that:
\begin{displaymath}
\int_{\mathbb{R}^d}\vert \mu(B(x,r))\vert^{\alpha}dx\leq C\left( r^{s}\wedge r^{t}\right),
\end{displaymath}
where $a\wedge b=\min(a,b)$.
\end{Def}

This definition is reminiscent of $\mb$ in \cite{Jcb}. The next proposition gathers some elementary properties that are easily generalized from that of $\mb$ in \cite{Jcb} (see the proof in the Appendix below).
\begin{Prop}  \label{propm}  \qquad
\begin{enumerate}
\item  \label{itlin}The set $\mcal$ is a linear subspace. Furthermore, if $g\in L^{\infty}(\Rd)$ and $\beta$ is a real-valued function on $\Rd$ such that $\forall x\in \Rd, \: \beta_1\leq \beta(x)\leq \beta_2$, then: 
\begin{equation} \label{maj}
\forall \mu \in \mcal,\: \I \vert \mu(B(x,r)) \vert^{\alpha}g(x)r^{-\beta(x)-1}dxdr <+\infty.
\end{equation} 
\item When $\alpha \leq \alpha'$, $\ma\subset\maa$.\label{itinc}
\item When $\beta_1\leq \beta_1'\leq \beta'_2\leq \beta_2$, $\ma\subset \mathcal{M}_{\alpha, \beta_1', \beta_2'}$.
\item When $\alpha \leq \alpha'$ and $\beta_1\leq \beta_1'\leq \beta'_2\leq \beta_2$, $\ma \subset \mathcal{M}_{\alpha', \beta_1', \beta_2'}$.
\item When $\beta_1>d$, $\mcal$ is included in the of subspace of diffuse measures.\label{itdif}
\end{enumerate}
\end{Prop}

\bigskip

The next Proposition is a direct application of Proposition 2.2 in \cite{Jcb} and gives some elegant properties of spaces $\mcal$:

\begin{Prop}\label{propgm}\quad
\begin{enumerate} 
\item The set $\mcal$ is closed under rotations, i.e. $\forall \mu \in \mcal$, $\forall \Theta \in \mathcal{O}(\Rd)$, $\Theta \mu \in \mcal$ where, for any Borelian set $A$ of $\Rd$, $\Theta\mu(A)=\mu(\Theta^{-1}A)$.
\item $\mcal$ is closed under dilatations, i.e. when $\mu\in \mcal$, for any $a\in \mathbb{R}^+$, $\mu_a\in \mcal$ where, for any Borelian subset $A$ of $\Rd$, $\mu_a(A)=\mu(a^{-1}A)$.
\end{enumerate}
\end{Prop}

\bigskip

Observe that Dirac measures $\delta_y$, when $ y\in Supp(g)$, are not in $\mcal$. Absolutely continuous measures (with respect to the Lebesgue measure) $\mu(dx)=\phi(x)dx$ with $\phi \in L^1(\Rd)\cap L^{\alpha }(\Rd)$ will play a crucial role in the small-balls scaling described below. For such $\mu$, we shall abusively note $\mu \in  L^1(\Rd)\cap L^{\alpha }(\Rd)$. As a direct application of Proposition 2.3 in \cite{Jcb} we have:
\begin{Prop}
When $d<\beta_1\leq \beta_2<\alpha d$, any measure $\mu \in  L^1(\Rd)\cap L^{\alpha }(\Rd)$ belongs to $\mcal$.
\end{Prop}

\bigskip

The following elementary Lemma will be useful at several instance. It is proved in the Appendix.
\begin{Lem} \label{lem:majf}
Let $t>d$. For any $A>0$:
\begin{displaymath}
\int_0^A r^t \Vert f(\cdot, r)\Vert_\infty dr <+\infty.
\end{displaymath}
\end{Lem}

We also recall the Lemma 3.1 of \cite{Jcb} which gives an estimate of the characteristic function of distribution in the domain of attraction of a stable law.
\begin{Lem}
Suppose $X$ is in the domain of attraction of an $\alpha$-stable law $S_\alpha(\sigma, \beta, 0)$ for some $\alpha>1$. Then
\begin{displaymath}\label{Lem:g}
\phi_X(\theta)-1-i\theta\mathbb{E}\left[X\right]\sim_0 -\sigma^{\alpha}\vert\theta\vert^\alpha\left(1-i\epsilon(\theta)\tan(\pi \alpha/2)b\right).
\end{displaymath}
Furthermore, there is some $C>0$ such that for any $\theta \in \mathbb{R}$, 
\begin{displaymath}
\left| \phi_X(\theta)-1-i\theta\mathbb{E}\left[X\right] \right| \leq K\vert \theta \vert^\alpha.
\end{displaymath}
\end{Lem}

\begin{Rq}
In the following sections, the convergences obtained are finite-dimensional convergences and $\stackrel{\mathcal{A}}{\longrightarrow}$ denotes the finite-dimensional convergence in a subspace $\mathcal{A}$ of $\Z$, i.e. $M_\rho \stackrel{\mathcal{A}}{\longrightarrow} M$ means
\begin{displaymath}
\forall n, \; \forall \mu_1, ...,\mu_n \in \mathcal{A}, \; \mathcal{L}\left(M_\rho(\mu_1),...,M_\rho(\mu_n)\right)\rightarrow \mathcal{L}\left(M(\mu_1),...,M(\mu_n)\right).
\end{displaymath}
\end{Rq}

\bigskip
\subsection{Large-balls scaling}

In this section, we study the fluctuations of $\m$ under the large-balls scaling, that is when $\lambda(\rho)\rho^{\beta_1}\to+\infty$, $\rho\to 0$. The limiting field obtained expresses as an $\alpha$-stable integral.

\bigskip

\begin{Thm}\label{thm:lbs}
Suppose that $\lambda(\rho)\rho^{\beta_1} \to + \infty$ when $\rho \to 0$. Let $n(\rho)=\lambda(\rho)^{1/\alpha} \rho^{\beta_1/\alpha}$ and suppose that $B_1=\left\lbrace x\in \Rd: \: \beta(x)=\beta_1 \right\rbrace$ has a non-zero Lebesgue measure, then we have:
\begin{displaymath}
n(\rho)^{-1}\Mr(\cdot) \stackrel{\mcal}{\longrightarrow} Z(\cdot),
\end{displaymath}
where $Z(\mu)=\I \mu(B(x,r))M_{\alpha}(dx,dr)$ is a stable integral with respect to the $\alpha$-stable measure $M_{\alpha}$ with control measure $\sigma^{\alpha}g(x)1_{B_1}(x)r^{-\beta_1-1}dxdr$ and constant skewness function $b$ depending on $G$.
\end{Thm}

\begin{Rq}
The stochastic integral $Z(\mu)$ is well defined for $\mu \in \mcal$ since $$\I \vert \mu(B(x,r)) \vert^{\alpha}g(x)1_{B_1}(x)r^{-\beta_1-1}dxdr<+\infty,$$ see \cite{ST}.
\end{Rq}

\begin{proof}

Our aim is to show that, for $\mu \in \mcal$, the $\log$-characteristic function of $n(\rho)^{-1}\Mr(\mu) $ given in \eqref{fcm} converges when $\rho \to 0$ to the $\log$-characteristic function of the random variable $Z(\mu)$: 
\begin{equation}\label{fcz}
\log\left(\varphi_{Z(\mu)}(\theta)\right)=-\sigma^{\alpha} \! \! \! \I \! \! \!  \vert \theta \mu(B(x,r))\vert^{\alpha}\left(1-i\epsilon(\theta\mu(B(x,r)))\tan(\pi\alpha/2)b\right)g(x)\frac{1_{B_1}(x)}{r^{\beta_1+1}}dxdr, 
\end{equation}
where $\epsilon(u)=+1$ if $u>0$, $\epsilon(u)=-1$ if $u<0$ and $\epsilon(0)=0$ (see \cite{ST}).
\\
Since by hypothesis, $n(\rho)\to +\infty$, we apply Lemma \eqref{Lem:g} which yields:
\begin{equation}\label{eqg}
\lambda(\rho)\psi_G\left( \theta n(\rho)^{-1}\mu (B(x,r))\right)\sim -\sigma^{\alpha}\rho^{-\beta_1}\vert \theta \mu(B(x,r))\vert^{\alpha} \bigg( 1-i \epsilon\big(\theta\mu(B(x,r))\big)\tan(\pi \alpha /2)b\bigg).
\end{equation}
\\
Since $\vert \theta n(\rho)^{-1}\mu(B(x,r))\vert\leq \theta n(\rho)^{-1}\vert\mu\vert(\mathbb{R}^d)$, the equivalence \eqref{eqg} is uniform both in $x$ and $r$ and can be integrated to obtain
\begin{equation}\label{pr1eq1}
\begin{split}
&\I \psi_G\left( n^{-1}(\rho)\theta\mu(B(x,r))\right)\rho^{-1}f_{\rho}(x,r/\rho)dxdr \\
\sim &-\sigma^{\alpha}\vert \theta \vert^{\alpha}\I \vert\mu(B(x,r))\vert^{\alpha}\left(1-i\epsilon(\theta\mu(B(x,r)))\tan(\pi\alpha/2)b\right)\frac{f_{\rho}(x,r/\rho)}{\rho n(\rho)^{\alpha}}dxdr,
\end{split}
\end{equation}
when $\rho \to 0$.

\bigskip

We now check that the right-hand side of \eqref{pr1eq1}  is equivalent as $\rho$ goes to $0$ to the $\log$-characteristic function \eqref{fcz} of $Z(\mu)$. But from \eqref{cfr} and by definition of $n(\rho)$, for all $x\in \Rd$ 

$$\frac{f_{\rho}(x,r/\rho)}{\rho n(\rho)^{\alpha}}\sim \rho^{\beta(x)-\beta_1}\frac{g(x)}{r^{\beta(x)+1}}\to \frac{g(x)1_{B_1}(x)}{r^{\beta_1+1}}, \; \rho \to 0$$ and we have to consider two cases apart: $x\in B_1$ and $x\notin B_1$.

\bigskip

First, we focus on the case $x\in B_1$. For big radii, we are able to replace the measure $\frac{f_{\rho}(x,r/\rho)}{\rho n(\rho)^{\alpha}}dxdr$ by its equivalent. Indeed, condition \eqref{cfr} can be written:
\begin{equation}\label{cfbis}
\forall \varepsilon>0, \: \exists A>0, \: \forall r>A, \: \forall x \in \Rd, \; \forall \rho >0, \: \bigg\vert \frac{f_\rho(x,r)}{\lambda(\rho)}-\frac{g(x)}{r^{\beta_1+1}}\bigg\vert \leq \varepsilon \frac{g(x)}{r^{\beta_1+1}}.
\end{equation}
A change of variable allows us to rewrite \eqref{cfbis} as follows
\begin{equation}\label{cfter}
\forall \varepsilon >0,\: \exists A>0,\: \forall \rho>0, \: \forall r>A\rho, \: \forall x \in B_1, \: \left| \frac{f_{\rho}(x,r/\rho)}{\rho n(\rho)^{\alpha}}- \frac{g(x)}{r^{\beta_1+1}}\right| \leq \varepsilon \frac{g(x)}{r^{\beta_1+1}}.
\end{equation}
Integrating \eqref{cfter} over $B_1$ yields
\begin{multline}\label{eq1}
\forall \varepsilon >0,\: \exists A>0,\: \forall \rho>0, \: \forall r>A\rho, \\
 \left| F_1(r,\rho)-F_2(r)\right|\leq \varepsilon (1-\tan(\pi\alpha/2)b)\int_{B_1}  \vert\mu(B(x,r))\vert^{\alpha}\frac{g(x)}{r^{\beta_1+1}}dx,
\end{multline}
where

$$F_1(r, \rho)=\int_{B_1} \! \! \vert\mu(B(x,r))\vert^{\alpha}\left(1-i\epsilon(\theta\mu(B(x,r)))\tan(\pi\alpha/2)b\right) \frac{f_{\rho}(x,r/\rho)}{\rho n(\rho)^{\alpha}}dx, $$ $$F_2(r)= \int_{B_1} \! \! \vert\mu(B(x,r))\vert^{\alpha}\left(1-i\epsilon(\theta\mu(B(x,r)))\tan(\pi\alpha/2)b\right)\frac{g(x)}{r^{\beta_1+1}}dx$$ and $$F_3(r)=\left(1-\tan(\pi\alpha/2)b\right)\int_{B_1}\vert\mu(B(x,r))\vert^{\alpha}\frac{g(x)}{r^{\beta_1+1}}dx.$$
\\
Since $\mu\in \mcal$ and $ g\in L^{\infty}(\Rd)$, using Point \ref{itlin} of Proposition \ref{propm} with $\beta(x)=\beta_1$:
\begin{equation}\label{ineq1}
\int_{\mathbb{R}^+}\vert F_2(r)\vert dr\leq\int_{\mathbb{R}^+}\vert F_3(r)\vert dr<+\infty .
\end{equation}
\\
From \eqref{eq1} and \eqref{ineq1}, we deduce:
\begin{equation}\label{ineqbr}
\forall \varepsilon>0, \: \exists A>0, \: \forall \rho >0,  \int_{A\rho}^{+\infty} \left| F_1(r,\rho)-F_2(r)\right|dr \leq \varepsilon\int_{\mathbb{R}^+}\vert F_3(r)\vert dr.
\end{equation}

 \bigskip

We show now that the integral over small radii is negligible. Let $\varepsilon>0$ and let $\rho >0$ such that $A\rho<1$ (where $A$ is given by \eqref{eq1}). Since $\mu \in \mcal$, there exist $t>\beta_1$ and a constant $C$ such that:
\begin{align}
\left|\int_0^{A\rho} F_1(r,\rho)dr\right|&\leq \big(1-b\tan(\pi \alpha/2)\big) \int_0^{A\rho} \int_{B_1}\vert \mu(B(x,r))\vert^{\alpha} \frac{f_{\rho}(x,r/\rho)}{\rho n(\rho)^{\alpha}}dxdr \notag \\ 
&\leq \frac{\big(1-b\tan(\pi \alpha/2)\big)}{\rho n(\rho)^{\alpha}} \int_0^{A\rho} \int_{B_1}\vert \mu(B(x,r))\vert^{\alpha}dx \Vert f_\rho(\cdot,r/\rho)\Vert_{\infty}dr \notag \\
&\leq \frac{\big(1-b\tan(\pi \alpha/2)\big)}{\rho^{\beta_1+1} \lambda(\rho)}C\int_0^{A\rho} r^t \lambda(\rho)\Vert f(\cdot,r/\rho)\Vert_{\infty}dr \notag \\
&\leq  \big(1-b\tan(\pi \alpha/2)\big) C \rho^{t-\beta_1} \int_0^A r^t\Vert f(\cdot,r)\Vert_{\infty}dr \label{majF1}\\
& \longrightarrow 0 \label{cvF1}
\end{align}
when $\rho \to 0$ using Lemma \ref{lem:majf} and \eqref{cfr1}.
\\
Thus, using \eqref{ineq1}, \eqref{ineqbr} and \eqref{cvF1}, we have
\begin{displaymath}
\begin{split}
\bigg\vert\int_{\mathbb{R}} F_1(r,\rho)dr-\int_{\mathbb{R}}F_2(r)dr\bigg\vert \leq &\int_{A\rho}^{+\infty} \left| F_1(r,\rho)-F_2(r)\right|dr + \left|\int_0^{A\rho} F_1(r,\rho)dr\right| +\left|\int_0^{A\rho} F_2(r)dr\right|\\
\leq  &\varepsilon\int_{\mathbb{R}^+}\vert F_3(r)\vert dr+\left|\int_0^{A\rho} F_1(r,\rho)dr\right| +\left|\int_0^{A\rho} F_2(r)dr\right|\\
\leq & K \varepsilon,
\end{split}
\end{displaymath}
for $\rho$ small enough. To conclude we have shown so far that:
\begin{displaymath}
\lim_{\rho \to 0} \int_{\mathbb{R}^+} F_1(r,\rho)dr= \int_{\mathbb{R}^+}F_2(r)dr.
\end{displaymath}
\\
We deal now with the case $x\notin B_1$ but apply the same arguments with slight changes. Set $$\widetilde{F}_1(r,\rho)=\int_{B_1^c}  \vert\mu(B(x,r))\vert^{\alpha}\bigg(1-i\epsilon\big(\theta\mu(B(x,r))\big)\tan(\pi\alpha/2)b\bigg)\frac{f_{\rho}(x,r/\rho)}{\rho n(\rho)^{\alpha}}dx,$$  $$\widetilde{F}_2(r,\rho)=\int_{B_1^c}\vert\mu(B(x,r))\vert^{\alpha}\bigg(1-i\epsilon\big(\theta\mu(B(x,r))\big)\tan(\pi\alpha/2)b\bigg)\frac{g(x)\rho^{\beta(x)-\beta_1}}{r^{\beta(x)+1}}dx$$ and $$\tilde{F}_3(r,\rho)=\big(1-b\tan(\pi\alpha/2)\big)\int_{B_1^c} \vert\mu(B(x,r))\vert^{\alpha}\frac{g(x)\rho^{\beta(x)-\beta_1}}{r^{\beta(x)+1}}dx$$ 
\\
For $\rho<1$, we have:
\begin{equation}\label{ineq2}
\left|\widetilde{F}_2(r,\rho)\right| \leq \left|\widetilde{F}_3(r,\rho)\right|\leq \big(1-b\tan(\pi\alpha/2)\big)\int_{\Rd} \vert\mu(B(x,r))\vert^{\alpha}g(x)r^{-\beta(x)-1}dx
\end{equation}
\\
and then, since $\mu\in\mcal$, using again Point \ref{itlin} of Proposition \ref{propm},  
\begin{align}\label{ineq2bis}
\int_{\mathbb{R}^+} \vert \widetilde{F}_2(r,\rho) \vert dr &\leq \int_{\mathbb{R}^+} \vert \widetilde{F}_3(r,\rho) \vert dr\notag \\
&\leq\big(1-b\tan(\pi\alpha/2)\big)\int_{\mathbb{R}^+}\int_{\Rd} \vert\mu(B(x,r))\vert^{\alpha}g(x)r^{-\beta(x)-1}dxdr<+\infty.
\end{align}
\\
Using Condition \eqref{cfr} with a change of variable and integrating over $\left[A\rho, +\infty\right[\times B_1^c$, we obtain:
\begin{displaymath}
\forall \varepsilon >0, \: \exists A>0, \: \forall \rho >0, \: \int_{A\rho}^{+\infty} \left| \widetilde{F}_1(r,\rho)-\widetilde{F}_2(r,\rho)\right|dr\leq \varepsilon \int_{\mathbb{R}^+}\vert \widetilde{F}_3(r,\rho) \vert dr.
\end{displaymath}
\\
Let $\varepsilon>0$ and let $\rho >0$ such that $A\rho<1$. Like in \eqref{majF1} above for $F_1$, we obtain
\begin{displaymath}
\begin{split}
\left|\int_0^{A\rho} \widetilde{F}_1(r,\rho)dr\right|&\leq \big(1-b\tan(\pi \alpha/2)\big) C \rho^{t-\beta_1}  \int_0^{ A} r^{t}\Vert f(\cdot,r)\Vert_{\infty}dr \\
&\longrightarrow 0, \; \rho \to 0\\
\end{split}
\end{displaymath}
using Lemma \ref{lem:majf}. This shows that the integral over small radii is negligible.
\\
Thus:
\begin{displaymath}
\begin{split}
&\bigg\vert\int_{\mathbb{R}} \widetilde{F}_1(r,\rho)dr-\int_{\mathbb{R}}\widetilde{F}_2(r,\rho)dr\bigg\vert \\
\leq &\int_{A\rho}^{+\infty} \left| \widetilde{F}_1(r,\rho)-\widetilde{F}_2(r,\rho)\right|dr + \left|\int_0^{A\rho} \widetilde{F}_1(r,\rho)dr\right| +\left|\int_0^{A\rho} \widetilde{F}_2(r,\rho)dr\right|\\
\leq & \varepsilon \int_{\mathbb{R}^+} \vert \widetilde{F}_3(r,\rho)\vert dr +\left|\int_0^{A\rho} \widetilde{F}_1(r,\rho)dr\right| +\left|\int_0^{A\rho} \widetilde{F}_2(r)dr\right|\\
\leq &K \varepsilon,
\end{split}
\end{displaymath}
for $\rho$ small enough. As a consequence, $\int_{\mathbb{R}^+}\widetilde{F}_1(r,\rho) dr$ and $\int_{\mathbb{R}^+}\widetilde{F}_2
(r,\rho)dr$ have the same limit when $\rho \to 0$ and we show now that this limit is $0$. Since $x\in B_1^c$, \begin{displaymath}
\lim_{\rho\to 0}(B(x,r))\vert^{\alpha}\bigg(1-i\epsilon\big(\theta\mu(B(x,r))\big)\tan(\pi\alpha/2)b\bigg)g(x)\frac{\rho^{\beta(x)-\beta_1}}{r^{\beta(x)+1}}= 0.
\end{displaymath}
Moreover, for $\rho<1$ 
\begin{multline}\label{dom}
\vert\mu(B(x,r))\vert^{\alpha}\bigg(1-i\epsilon\big(\theta\mu(B(x,r))\big)b\tan(\pi\alpha/2)\bigg)g(x)\frac{\rho^{\beta(x)-\beta_1}}{r^{\beta(x)+1}}\\
\leq \big(1-b\tan(\pi\alpha /2)\big) \vert\mu(B(x,r)))\vert^{\alpha}g(x)r^{-\beta(x)-1}.
\end{multline}
But for $\mu \in \mcal$, Proposition \eqref{propm}.\eqref{itlin} ensures that the right-hand side in \eqref{dom} is integrable over $\mathbb{R}^+\times B_1^c$. Dominated convergence Theorem entails
\begin{displaymath}
\lim_{\rho \to 0}=\int_{\mathbb{R}^+}\int_{B_1^c} \vert\mu(B(x,r))\vert^{\alpha}\left(1-i\epsilon(\theta\mu(B(x,r)))\tan(\pi\alpha/2)b\right)g(x)\frac{\rho^{\beta(x)-\beta_1}}{r^{\beta(x)+1}} dxdr= 0.
\end{displaymath}

\bigskip

Combining it all, we have proved that the $\log$-characteristic function of $n(\rho)^{-1}\Mr(\mu) $ converges to the $\log$-characteristic function of $Z(\mu)$ given in $\eqref{fcz}$. 

\bigskip

We now use the Cramér-Wold device and the linearity of the fields $\m$ and $Z$ to derive the convergence of the finite-dimensional distributions from the one dimensional convergence.

\end{proof}

\bigskip

In the next proposition, we give properties of the limit field $Z$ for particular choices of $g$ and $\beta$. They are deduced from the invariance by rotation of the Lebesgue measure, the self-similarity of stable integral, the global invariance of the balls and Proposition \ref{propgm}.

\begin{Prop} \quad \label{propZ}
\begin{enumerate}
\item \label{propZ1} When $g$ is radial (i.e. $g(x)=g(\Vert x\Vert)$) and $B_1$ is invariant by rotation, the field $Z$ is isometric, i.e.:
\begin{displaymath}
\forall \mu \in \mcal, \: \forall \Theta \in \mathcal{O}(\Rd), \: Z(\Theta \mu)\stackrel{fdd}{=} Z(\mu).
\end{displaymath}
\item \label{propZ2} Suppose there exists $H\in \mathbb{R}$ such that, for any $a \in \mathbb{R}^+$ and $x\in \Rd$, $g(ax)=a^Hg(x)$ and suppose $B_1$ is invariant by dilatation (i.e $\left\lbrace x\in \Rd, \: ax\in B_1\right\rbrace=B_1$). Then, the random field $Z$ is self-similar on $\mcal$ with index $(H+d-\beta_1)/\alpha$.
\end{enumerate}
\end{Prop}
\begin{proof} \quad
\begin{enumerate}
\item The proof comes directly from the invariance by rotation of the measure  $g(x)1_{B_1}(x)dx$ which is the projection of the control measure of $M_{\alpha}$ on $\Rd$.
\item The proof consists in a change of variable in \eqref{fcz}.
\end{enumerate}
\end{proof}
\begin{Rq}\quad
\begin{itemize}
\item The only way to have $B_1$ invariant by rotation and by dilatation is to take $B_1=\Rd$ (i.e. to take $\beta$ a constant function). 
\item Let $g$ be a radial function and such that, for any $a\in \mathbb{R}^+$, for any $x\in \Rd$, $g(ax)=a^Hg(x)$, then $g(x)=\Vert x\Vert^H g(1)$. Unfortunately, $g_H:x\mapsto \Vert x\Vert^{H}$ belongs to $L^\infty$ implies $H=0$ (i.e. $g$ is a constant function). In that case, we recover the settings of \cite{Jcb}.
\end{itemize}
\end{Rq}

\bigskip

\begin{Rq}
To study the spatial dependence structure of the process $Z$, we use the covariation which is a generalization of the covariance to the stable framework (see \cite{ST}). Let $\mu_1,\mu_2\in \mcal$, the covariation of $Z(\mu_1)$ and $Z(\mu_2)$ is given by:
\begin{displaymath}
\begin{split}
&\left[Z(\mu_1),Z(\mu_2)\right]_{\alpha} \\
= & \sigma^{\alpha}  \I \mu_1(B(x,r)) \epsilon\left( \mu_2(B(x,r))\right)\left| \mu_2(B(x,r))\right|^{\alpha-1}g(x)1_{B_1}(x)r^{-\beta_1-1}dxdr.
\end{split}
\end{displaymath}
The integral above is well defined by Hölder's inequality. Since, even when $\mu_1$ and $\mu_2$ have disjoint supports, we have $\left[Z(\mu_1),Z(\mu_2)\right]_{\alpha}\neq 0$, this indicates that $Z(\mu_1)$ and $Z(\mu_2)$ are stable-dependent.
\end{Rq}

\bigskip

\begin{Rq}\label{rq:lbs}
When $\beta$ is a constant function and $g=1$, the field $Z$ coincides with the stable random field obtained in Theorem 2.4 in \cite{Jcb}. Therefore, referring to Remarks 2.8 and 2.1 in \cite{Jcb}, it generalizes the Telecom Process obtained in \cite{KT}, a Gaussian limit field obtained in \cite{Bek} and the large-grain limit in \cite{Kaj}.
\end{Rq}

\bigskip

\subsection{Intermediate scaling}
In this section, we investigate the intermediate scaling, that is when $\lambda(\rho)\rho^{\beta_1}$ has a finite non-zero limit $\ell$ when $\rho \to 0$. In this case, the field obtained at the limit is a compensated Poisson integral.
\begin{Thm}\label{thm:is}
Suppose $\lambda(\rho)\rho^{\beta_1} \to \ell\in \left]0, +\infty \right[$ when $\rho \to 0$ and suppose again \\the set $B_1=\left\lbrace x\in \Rd: \: \beta(x)=\beta_1 \right\rbrace$ has a non-zero Lebesgue measure, then we have:
\begin{displaymath}
\Mr(\cdot) \stackrel{\mcal}{\longrightarrow} J_\ell(\cdot), 
\end{displaymath}
where $J_\ell(\mu)=\II m \mu(B(x,r)) \widetilde{\Pi}_\ell(dx,dr,dm)$ and $\widetilde{\Pi}_\ell$ is a compensated Poisson random measure on $\mathbb{R}^d\times \mathbb{R}^+\times \mathbb{R}$ with intensity $\ell g(x)1_{B_1}(x)dx r^{-\beta_1-1}drG(dm)$.
\end{Thm}

\bigskip

\begin{Rq}
In that particular case, the normalization factor $n(\rho)$ is $1$ and, roughly speaking, $J$ is obtained by taking the limit in the intensity measure.
\end{Rq}

\bigskip

\begin{proof}
We first prove that the compensated Poisson integral $J_\ell(\mu)$ is well defined for $\mu \in \mcal$. This is the case when
\begin{equation}\label{c:defJ}
\II \bigg( \big\vert m\mu(B(x,r))\big\vert \wedge \big(m\mu(B(x,r))\big)^2\bigg)g(x)1_{B_1}(x)dx r^{-\beta_1-1}drG(dm)<+\infty,
\end{equation}
see  Lemma 12.13 in \cite{Kal}. Condition \eqref{c:defJ} can be splitted into:
\begin{equation}
\label{c:def1} 
\int_{\vert m\mu(B(x,r))\vert\leq 1}(m\mu(B(x,r)))^2g(x)1_{B_1}(x)dx r^{-\beta_1-1}drG(dm)<+\infty
\end{equation}
and
\begin{equation}
\label{c:def2}
\int_{\vert m\mu(B(x,r))\vert\geq 1}\vert m\mu(B(x,r))\vert g(x)1_{B_1}(x)dx r^{-\beta_1-1}drG(dm)<+\infty.
\end{equation}
To prove \eqref{c:def1} and \eqref{c:def2}, we use the following bounds for truncated moment of $G$:
\begin{displaymath}
\int_{\vert m\vert\geq x}\vert m\vert G(dm) \leq C_1 x^{1-\alpha}
\end{displaymath}
and 
\begin{displaymath}
\int_{-x}^x m^2 G(dm) \leq C_2 x^{2-\alpha},
\end{displaymath}
for all $x\geq 0$ and for some constants $C_1$ and $C_2$, see Lemma 3.4 in \cite{Jcb}. Therefore, for \eqref{c:def1} we have:
\begin{displaymath}
\begin{split}
&\int_{\vert m\mu(B(x,r))\vert\leq 1}(m\mu(B(x,r)))^2g(x)1_{B_1}(x)dx r^{-\beta_1-1}drG(dm) \\
= & \I \left(\int_{-1/\vert \mu(B(x,r))\vert}^{1/\vert \mu(B(x,r))\vert} m^2 G(dm)\right)\vert \mu(B(x,r))\vert^{2}g(x)1_{B_1}(x)dx r^{-\beta_1-1}dr \\
\leq & C_2 \I \vert\mu(B(x,r))\vert^{\alpha}g(x)r^{-\beta_1-1}dxdr
\end{split}
\end{displaymath}
which is finite when $\mu \in \mcal$ thank to Proposition \ref{propm}. Similarly, for \eqref{c:def2} we have:
\begin{displaymath}
\begin{split}
&\int_{\vert m\mu(B(x,r))\vert\geq 1}\vert m\mu(B(x,r))\vert g(x)1_{B_1}(x)dx r^{-\beta_1-1}drG(dm)\\
= & \I \left( \int_{\vert m\vert\geq 1/\vert \mu(B(x,r))\vert} \vert m \vert G(dm)\right) \vert \mu(B(x,r)) \vert g(x)1_{B_1}(x)dx r^{-\beta_1-1}dr \\
\leq & C_1 \I \vert\mu(B(x,r))\vert^{\alpha} g(x) r^{-\beta_1-1}dxdr<+\infty.
\end{split}
\end{displaymath}

\bigskip

We prove now Theorem \ref{thm:is} using a very similar reasoning as the one used for \\Theorem \ref{thm:lbs}. Condition \eqref{cfr} implies that $\rho^{-1}f_\rho(x,r/\rho )\sim \lambda(\rho)g(x) \rho^{\beta(x)}r^{-\beta(x)-1}$ when $\rho \to 0$. Since $\lambda(\rho) \rho^{\beta(x)}r^{-\beta(x)-1} \to \ell r^{-\beta_1-1}1_{B_1}(x)$ when $\rho \to 0$, we consider two cases: $x\in B_1$ and $x\neq B_1$.
\\
First, we focus on the case $x\in B_1$. Set $$G_1(r,\rho)=\int_{B_1}  \psi_G(\theta \mu(B(x, r))) \rho^{-1}f_\rho(x,r/\rho )dx,$$
 $$G_2(r)=\ell\int_{B_1} \psi_G(\theta \mu(B(x, r))) g(x)  r^{-\beta_1-1} dx$$ and $$G_3(r)=\ell \int_{B_1}\left| \psi_G(\theta \mu(B(x, r)))\right|g(x)  r^{-\beta_1-1} dx.$$
\\
A change of variable in \eqref{cfbis} entails
\begin{equation}\label{étoile}
\forall \varepsilon >0, \: \exists A>0, \: \forall \rho>0,\: \int_{A\rho}^{+\infty} \left| G_1(r,\rho)-G_2(r)\right|dr \leq \varepsilon \int_{\mathbb{R}^+}\vert G_3(r)\vert dr.
\end{equation}
We use the fact that $\left| \psi_G(\theta \mu(B(x, r)))\right| \leq K\vert \theta \mu(B(x, r))\vert^{\alpha}$ (Lemma \eqref{Lem:g}) to derive that:
\begin{equation} \label{ineq3}
\left|G_2(r)\right|\leq \vert G_3(r)\vert \leq \ell K\vert\theta\vert^{\alpha}\int_{B_1}\vert\mu(B(x, r))\vert^{\alpha} g(x)dx r^{-\beta_1-1}
\end{equation}
and, since $\mu \in \mcal$, Proposition \ref{propm} ensures
\begin{equation}\label{ineq3bis}
\int_{\mathbb{R}^+}\vert G_2(r) \vert dr \leq \int_{\mathbb{R}^+}\vert G_3(r) \vert dr \leq K \ell \vert\theta\vert^{\alpha}\int_{\mathbb{R}^+} \int_{B_1}\vert\mu(B(x, r))\vert^{\alpha} g(x)dx r^{-\beta_1-1}dr<+\infty.
\end{equation}
\\
On the other hand, let $\rho$ such that $A\rho<1$. Then, using \eqref{cfr1}, since $\mu \in \mcal$, there exist $t>\beta_1$ and a constant $C$ such that:
\begin{align}
\int_0^{A\rho} \left| G_1(r,\rho)\right| dr &\leq  \rho^{-1} \vert \theta \vert^{\alpha}\int_0^{A\rho} \int_{B_1} \vert\mu(B(x, r))\vert^{\alpha} g(x)dx \Vert f_\rho(\cdot ,r/\rho)\Vert_{\infty}dr \notag\\
&\leq C\lambda(\rho) \rho^{-1} \vert \theta \vert^{\alpha}\int_0^{A\rho}  r^t \Vert f(\cdot ,r/\rho)\Vert_{\infty}dr \notag\\
&\leq C\lambda(\rho) \rho^{t} \vert \theta \vert^{\alpha}\int_0^{A}  r^t \Vert f(\cdot ,r)\Vert_{\infty}dr \notag\\
&\to 0 \label{cvG_1}
\end{align}
because $\int_0^{A}  r^t \Vert f(\cdot ,r)\Vert_{\infty}dr$ comes from Lemma \ref{lem:majf} and $\lambda(\rho)\rho^t=\lambda(\rho)\rho^{\beta_1}\rho^{-\beta_1+t}\to 0$ when $\rho \to 0$ since $\lambda(\rho)\rho^{\beta_1}\to \ell $. In conclusion, \eqref{étoile}, \eqref{ineq3bis} and \eqref{cvG_1} entail

\begin{displaymath}
\begin{split}
\bigg\vert\int_{\mathbb{R}} G_1(r,\rho)dr-\int_{\mathbb{R}}G_2(r)dr\bigg\vert &\leq \int_{A\rho}^{+\infty} \! \! \! \left| G_1(r,\rho)-G_2(r)\right|dr + \left|\int_0^{A\rho} G_1(r,\rho)dr\right| +\left|\int_0^{A\rho} G_2(r)dr\right|\\
&\leq \varepsilon \int_{\mathbb{R}^+} \vert G_3(r)\vert dr +\left|\int_0^{A\rho} G_1(r,\rho)dr\right| +\left|\int_0^{A\rho} G_2(r)dr\right|\\
&\leq K \varepsilon,
\end{split}
\end{displaymath}
for $\rho$ small enough.
\\
We are now interested in the case $x\in B_1^c$. Set $$\widetilde{G}_1(r,\rho)=\int_{B_1^c}  \psi_G(\theta \mu(B(x, r)))  \rho^{-1}f_\rho(x,r/\rho )dx$$, $$\widetilde{G}_2(r,\rho)= \int_{B_1^c} \psi_G(\theta \mu(B(x, r))) g(x)  \frac{  \lambda(\rho)\rho^{\beta(x)}}{r^{\beta(x)+1}} dx $$ and
 $$\widetilde{G}_3(r,\rho)= \int_{B_1^c} \vert\psi_G(\theta \mu(B(x, r))) \vert g(x)  \frac{  \lambda(\rho)\rho^{\beta(x)}}{r^{\beta(x)+1}} dx .$$
 Note that, $\forall x \in B_1^c$, $\lambda(\rho) \rho^{\beta(x)}\leq \lambda(\rho)\rho^{\beta_1}$ when $\rho \leq 1$ and $\lambda(\rho)\rho^{\beta_1}\to \ell$ when $\rho \to 0$. Therefore, there exists $0<\delta\leq 1$ such that $\forall \rho <\delta$, $\lambda(\rho)\rho^{\beta(x)}<2\ell$. Let $\rho <\delta$.  Using the fact that $\left| \psi_G(\theta \mu(B(x, r)))\right| \leq K\vert \theta \mu(B(x, r))\vert^{\alpha}$ and Proposition \ref{propm} we have:
\begin{align}\label{ineq4}
\left| \int_{A\rho}^{+\infty} \widetilde{G}_2(r,\rho)  dr\right| 
 &\leq \left| \int_{A\rho}^{+\infty} \widetilde{G}_3(r,\rho)  dr\right| \notag \\
 &\leq 2\ell K\vert \theta \vert^{\alpha}\int_{A\rho}^{+\infty} \int_{B_1^c} \vert \mu (B(x, r))\vert^{\alpha} g(x)dx r^{-\beta(x)-1}dr<+\infty.
\end{align}
Then using Condition \eqref{cfr} with a change of variable we deduce:
\begin{equation}\label{ineq4bis}
\forall \varepsilon >0, \: \exists A>0,\: \forall \rho >0, \: \left| \int_{A\rho}^{+\infty}\widetilde{G}_1(r,\rho)dr-\int_{A\rho}^{+\infty}\widetilde{G}_2(r,\rho)dr\right| \leq \varepsilon \int_{A\rho}^{+\infty} \vert \widetilde{G}_3(r,\rho) \vert dr.
\end{equation}
In the mean time, there exist $t>\beta_1$ and a constant $C$ such that, using the same computations we used for $G_1$:
\begin{equation}\label{cvTG}
 \left|\int_0^{A\rho}\widetilde{G}_1(r,\rho)dr\right| \leq   K\vert \theta \vert^{\alpha} \lambda(\rho)\rho^{t}  C \int_0^A r^t \Vert f(\cdot, r)\Vert_{\infty}dr
\longrightarrow 0,
\end{equation}
when $\rho \to 0$ using Lemma \ref{lem:majf}.
Using \eqref{ineq4}, \eqref{ineq4bis} and \eqref{cvTG}, we have
\begin{displaymath}
\begin{split}
&\bigg\vert\int_{\mathbb{R}} \widetilde{G}_1(r,\rho)dr-\int_{\mathbb{R}}\widetilde{G}_2(r,\rho)dr\bigg\vert \\
\leq &\int_{A\rho}^{+\infty} \left| \widetilde{G}_1(r,\rho)-\widetilde{G}_2(r,\rho)\right|dr + \left|\int_0^{A\rho} \widetilde{G}_1(r,\rho)dr\right| +\left|\int_0^{A\rho} \widetilde{G}_2(r,\rho)dr\right|\\
\leq &K \varepsilon,
\end{split}
\end{displaymath}
for $\rho$ small enough.
Thus $$\int_{\mathbb{R}^+}\int_{B_1^c}  \psi_G(\theta \mu(B(x, r))) \rho^{-1}f_{\rho}(x,r/\rho ) dxdr$$ has the same limit as $$\int_{\mathbb{R}^+}\int_{B_1^c} \psi_G(\theta \mu(B(x, r))) g(x)  \lambda(\rho) \rho^{\beta(x)}r^{-\beta(x)-1} dxdr$$ when $\rho \to 0$.We now show that this limit is $ 0$. Since $x\in B_1^c$, $$ \psi_G(\theta \mu(B(x, r))) g(x)  \lambda(\rho) \rho^{\beta(x)}r^{-\beta(x)-1} \to 0,$$ when $\rho \to 0$ because $\lambda(\rho)\rho^{\beta(x)}=\lambda(\rho)\rho^{\beta_1}\rho^{\beta(x)-\beta_1}$ and $\lambda(\rho)\rho^{\beta_1}\to \ell$. Moreover, for $\rho<\delta$:
\begin{displaymath}
\left| \psi_G(\theta \mu(B(x, r))) g(x)  \lambda(\rho) \rho^{\beta(x)}r^{-\beta(x)-1} \right|\leq  2\ell K \vert \theta \vert^{\alpha} \vert \mu(B(x, r))\vert^{\alpha} g(x)r^{-\beta(x)-1},
\end{displaymath}
which is integrable on $\mathbb{R}^+\times B_1^c$ according to Proposition \ref{propm}. Thus by the dominated convergence theorem: 
\begin{displaymath}
\lim_{\rho \to 0}\int_{\mathbb{R}^+}\int_{B_1^c}  \psi_G(\theta \mu(B(x, r)))g(x) \lambda(\rho) \rho^{\beta(x)}r^{-\beta(x)-1} dxdr= 0, 
\end{displaymath}
which implies that:
\begin{displaymath}
\lim_{\rho\to 0}\int_{\mathbb{R}^+}\int_{B_1^c}  \psi_G(\theta \mu(B(x, r))) \rho^{-1}f_{\rho}(x,r/\rho ) dxdr= 0.
\end{displaymath}
\\
This proves that the $\log$-characteristic function of $\Mr(\mu)$ converges to $$\ell \I \psi_G\left( \theta \mu(B((x,r)))\right) \frac{g(x)1_{B_1}(x)}{r^{\beta_1+1}}dxdr$$ which is the $\log$-characteristic function of $J_l(\mu)$. Once again, using the Cramér-Wold device, this one-dimensional convergence is enough to prove the finite-dimensional distributions convergence of the process.

\end{proof}

\bigskip

We give now some remarkable properties of the limit field $J_\ell$ for particular choices of $g$ and $\beta$.
\begin{Prop} \quad
\begin{enumerate}
\item When $g$ is radial and $B_1$ is invariant by rotation, the field $J_\ell$ is isotropic.
\item When $g$ is such that there exists a constant $H>0$, such that for any $a\in \mathbb{R}^+$ and for any $x\in \Rd$, $g(ax)=a^Hg(x)$, and $B_1$ is invariant by dilatation, $J_\ell(\mu)$ is equal in finite-dimensional law to $J'(\mu_{\ell'})$ where $\ell'=\ell^{1/(d+H-\beta_1)}$ and, for any $\mu \in \mcal$:
\begin{displaymath}
J'(\mu)=\II m\mu(Bx,r) \widetilde{\Pi}(dx,dr,dm),
\end{displaymath}
where $\widetilde{\Pi}$ is a compensated Poisson process with intensity $ g(x)1_{B_1}(x)dx\frac{dr}{r^{\beta_1+1}}G(dm)$.
In that particular case, $J'$ is aggregate-similar that is, for any $\mu \in \mcal$, for any $m\geq 1$:
\begin{displaymath}
J'(\mu_{a_m})\stackrel{fdd}{=} \sum_{i=1}^m J'_i(\mu)
\end{displaymath}
where $J'_i$, $1\leq i\leq m$ are independent copies of $J'$ and $a_m=m^{1/(d-\beta_1-H)}$. 
\end{enumerate}
\end{Prop}

\begin{proof} \quad
\begin{enumerate}
\item The proof comes directly from the invariance by rotation of the measure 
$g(x)1_{B_1}(x)dx$ which is the projection on $\Rd$ of the control measure of $\Pi_\ell$.
\item Let $\mu\in \mcal$ and $\theta \in \mathbb{R}$, using conditions on $g$ and $B_1$ we have:
\begin{displaymath}
\begin{split}
\log\left(\mathbb{E}\left[e^{i\theta J'(\mu_{\ell'})}\right]\right)&=  \I \psi_G\left( \theta\mu(B(\ell'^{-1}x, \ell'^{-1}r))\right) g(x)1_{B_1}(x)dxr^{-\beta_1-1}dr \\
&=\ell'^{d-\beta_1} \I \psi_G\left(\theta\mu(B(x,r))\right) g(\ell'x)1_{B_1}(\ell'x)dxr^{-\beta_1-1}dr \\
&=\ell'^{d+H-\beta_1} \I \psi_G\left(\theta\mu(B(x,r))\right) g(x)1_{B_1}(x)dxr^{-\beta_1-1}dr \\
&=\log\left(\mathbb{E}\left[e^{i\theta J(\mu)}\right]\right).
\end{split}
\end{displaymath}
This proves the identity in law of $( J(\mu))_\mu$ and $(J'(\mu_{\ell'}))_\mu$.
Take now $\ell'=m^{1/(d-\beta_1-H)}$, where $m$ is a positive integer, in the previous computation (which means taking $\ell=m$ a positive integer) and we get the aggregate-similarity property.
\end{enumerate}
\end{proof}

\begin{Rq}\label{rq:is}
As in Remark \ref{rq:lbs}, the process $J_\ell$ defined in Theorem \ref{thm:is} recovers the process $J$ of Theorem 2.11 in \cite{Jcb} when $g=1$ and $\beta$ is a constant function.
\end{Rq}

\subsection{Small-balls scaling}

In this section, we investigate the asymptotic behaviour of $M$ under the small-balls scaling, that is $\lambda(\rho) \rho^{\beta_1}\to 0$, $\rho \to 0$. In that case, we obtain at the limit a stable-field that exhibits independence.
\begin{Thm}\label{thm:sbs}
Let $n(\rho)=(\lambda(\rho)^{1/\beta_1}\rho)^d$ and $\gamma=\beta_1/d \in \left]1,\alpha \right[$. Suppose $\lambda(\rho)\rho^{\beta_1} \to 0$ when $\rho \to 0$, $B_1=\left\lbrace x\in \Rd: \beta(x)=\beta_1\right\rbrace$ has a non-zero Lebesgue measure and $\beta_2<\alpha d$, then:
\begin{displaymath}
n(\rho)^{-1}\Mr(\cdot) \stackrel{L^1(\mathbb{R}^d)\cap L^{\alpha }(\mathbb{R}^d)}{\longrightarrow} \widetilde{Z}(\cdot),
\end{displaymath}
where, for $\mu(dx)=\phi(x)dx$, $\widetilde{Z}(\mu)=\int_{\mathbb{R}^d}\phi(x) M_{\gamma}(dx)$ is a stable integral with respect to the $\gamma$-stable measure $M_{\gamma}$ with control measure $\sigma_{\gamma}^{\gamma}1_{B_1}(x)g(x)dx$ for 
\begin{displaymath}
\sigma_{\gamma}^{\gamma}=\frac{c_d^{\gamma}}{d}\int_{\mathbb{R}^+}\frac{1-cos(r)}{r^{\gamma+1}}dr\int_{\mathbb{R}}\vert m\vert^{\gamma}G(dm)
\end{displaymath}
and with constant skewness function equals to 
\begin{equation} \label{bgamma}
b_{\gamma}=-\frac{\int_{\mathbb{R}}\epsilon(m)\vert m\vert^{\gamma}G(dm)}{\int_{\mathbb{R}}\vert m\vert^{\gamma}G(dm)}.
\end{equation}
\end{Thm}

\begin{Rq}
The limiting field $\widetilde{Z}$ is well defined since $d<\beta_1<\alpha d$ (see \cite{ST}). The condition $\beta_2<\alpha d$ implies that the volumes of the balls have an infinite variance. In other world, we need some balls to be big enough in order to obtain something significant at the limit.
\end{Rq}

\begin{proof}
Here again, we only prove the one-dimensional convergence. Indeed, combined with the linearity of our processes and the linear structure of $L^1(\mathbb{R}^d)\cap L^{\alpha }(\mathbb{R}^d)$, it implies the finite-dimensional convergence.

\bigskip

We make the change of variable $r\mapsto n(\rho)^{1/d}r$ in the log-characteristic function of $\frac{\Mr(\mu)}{n(\rho)}$ which gives us:
 \begin{displaymath}
 \begin{split}
 \I \psi_G&\left(\theta n(\rho)^{-1}\mu(B(x,n(\rho)^{1/d}r))\right)\frac{n(\rho)^{1/d}}{\rho}f_\rho\left(x,\frac{n(\rho)^{1/d}}{\rho}r\right)dxdr\\
 &=\I \psi_G \left(\theta n(\rho)^{-1}\mu(B(x,n(\rho)^{1/d}r))\right)\lambda(\rho)^{1/\beta_1} f_\rho\left(x,\lambda(\rho)^{1/\beta_1} r\right)dxdr.
\end{split} 
\end{displaymath}
\\
Note that under the assumption $\lambda(\rho)\rho^{\beta_1} \to 0$ when $\rho \to 0$ and Condition \eqref{cbeta}, $n(\rho)\to 0$ and $n(\rho)^{1/d}/\rho=\lambda(\rho)^{1/\beta_1} \to +\infty$ when $\rho \to 0$. Let $\mu(dz)=\phi(z)dz$ with $\phi \in L^1(\mathbb{R}^d)\cap L^{\alpha}(\mathbb{R}^d)$, then from Lemma 4 in \cite{Kaj}, when $n(\rho)\to 0$, 
\begin{equation} \label{cvphi}
\left| c_d^{-1}r^{-d}n(\rho)^{-1}\mu(B(x,n(\rho)^{1/d}r))- \phi(x)\right|\to 0
\end{equation}
$dx$-almost everywhere and $\phi^* \in L^{\alpha}(\mathbb{R}^d)$ where
\begin{equation}\label{cphi}
\phi^*(x)=\sup_{v>0}\bigg(c_d^{-1}v^{-d}\vert \mu \vert\big(B(x,v)\big)\bigg) .
\end{equation}
Using the Lipschitzian property of $\psi_G$ we have:
\begin{equation}\label{cvpsi}
\left| \psi_G\left(\theta \frac{\mu(B(x,n(\rho)^{1/d}r))}{n(\rho)}\right)-\psi_G(\theta\phi(x)c_dr^d)\right|\leq L \vert\theta \vert \left|\frac{\mu(B(x,n(\rho)^{1/d}r))}{n(\rho)} - \phi(x)c_dr^d\right|, 
\end{equation}
where $L=2\int_{\mathbb{R}}\vert m\vert G(dm)$, and thus $\left| \psi_G\left(\theta n(\rho)^{-1}\mu(B(x,n(\rho)^{1/d}r))\right)-\psi_G(\theta\phi(x)c_dr^d)\right|$ converges to $0$ when $\rho \to 0$.
We also have from \eqref{cfr} that:
\begin{displaymath}
\lim_{\rho \to 0} \lambda(\rho)^{1/\beta_1} f_\rho\left(x,\lambda(\rho)^{1/\beta_1} r\right)= \frac{g(x)1_{B_1}(x)}{r^{\beta(x)+1}}.
\end{displaymath}
Thus:
\begin{equation}\label{CV}
\begin{split}
&\lim_{\rho\to 0}\psi_G\left(\theta n(\rho)^{-1}\mu(B(x,n(\rho)^{1/d}r))\right)\lambda(\rho)^{1/\beta_1} f_\rho\left(x,\lambda(\rho)^{1/\beta_1} r\right)\\
= &\psi_G(\theta\phi(x)c_dr^d)\frac{g(x)1_{B_1}(x)}{r^{\beta(x)+1}}.
\end{split}
\end{equation}
We want to show that we can exchange the limit \eqref{CV} with the integration over $\mathbb{R}^+\times \Rd$. To see this, we write, $\forall r\in \mathbb{R}^+$, $\forall x\in \Rd$:
 \begin{multline}\label{majH}
 \left|\psi_G\left(\theta n(\rho)^{-1}\mu(B(x,n(\rho)^{1/d}r))\right)\lambda(\rho)^{1/\beta_1} f_\rho\left(x,\lambda(\rho)^{1/\beta_1} r\right)- \psi_G(\theta\phi(x)c_dr^d)\frac{g(x)1_{B_1}(x)}{r^{\beta(x)+1}}\right|\\
  \quad \leq H_1(x,r,\rho) + H_2(x,r,\rho),
  \end{multline}
  where
  \begin{displaymath}
  H_1(x,r,\rho)=\left| \psi_G\left(\theta n(\rho)^{-1}\mu(B(x,n(\rho)^{1/d}r))\right)- \psi_G\left(\theta\phi(x)c_dr^d\right)\right|  \left| \lambda(\rho)^{1/\beta_1}f_\rho(x,\lambda(\rho)^{1/\beta_1}r)\right|
  \end{displaymath}
  and 
  \begin{displaymath}
  H_2(x,r,\rho)=\left| \psi_G(\theta\phi(x)c_dr^d)\right| \left|\lambda(\rho)^{1/\beta_1} f_\rho\left(x,\lambda(\rho)^{1/\beta_1} r\right)- \frac{g(x)1_{B_1}(x)}{r^{\beta(x)+1}}\right|.
  \end{displaymath}
 \\
 First, let us focus on $H_1$. Let $\rho>0$ large enough such that $\lambda(\rho) >1$, then, using \eqref{cvpsi} and the bound induced by Condition \eqref{cfr}, $\forall r>A\lambda(\rho)^{-1/\beta_1}$, $\forall x\in \Rd$:
\begin{align}
H_1(x,r,\rho)&\leq   L \vert\theta \vert \left|n(\rho)^{-1}\mu(B(x,n(\rho)^{1/d}r))- \phi(x)c_dr^d\right| 2g(x) \frac{\lambda(\rho)^{1-\beta(x)/\beta_1}}{r^{\beta(x)+1}} \notag \\
&\leq 2 L \vert\theta \vert  c_dr^d\left|c_d^{-1} r^{-d} n(\rho)^{-1}\mu(B(x,n(\rho)^{1/d}r))- \phi(x)\right|g(x) \left(r^{-\beta_1-1}\vee r^{-\beta_2-1}\right),
\end{align}
where $a\vee b=\max(a,b)$. \\
Moreover, $\forall r\in \mathbb{R}^+$
\begin{displaymath}
\begin{split}
\int_{\Rd} &\left|c_d^{-1} r^{-d} n(\rho)^{-1}\mu(B(x,n(\rho)^{1/d}r))- \phi(x)\right|g(x) dx \\
&\leq \Vert g\Vert_{\infty}\int_{\Rd} \left|c_d^{-1} r^{-d} n(\rho)^{-1}\mu(B(x,n(\rho)^{1/d}r))- \phi(x)\right|  dx.
\end{split}
\end{displaymath}
The integrand $\left|c_d^{-1} r^{-d} n(\rho)^{-1}\mu(B(x,n(\rho)^{1/d}r))- \phi(x)\right|$ converges to $0$ $dx$-almost everywhere (see \eqref{cvphi}). Since its $L^{\alpha}$-norm is bounded by $\Vert \phi^*\Vert_{L^{\alpha}}+\Vert \phi\Vert_{L^{\alpha}}$, it is uniformly integrable in $r$ and $\rho$ and as a consequence:
\begin{displaymath}
\lim_{\rho \to 0} \int_{\Rd} \left|c_d^{-1} r^{-d} n(\rho)^{-1}\mu(B(x,n(\rho)^{1/d}r))- \phi(x)\right|  dx =0.
\end{displaymath}
\\
In particular, 
\begin{multline}
\forall \varepsilon >0, \: \exists \rho_0>0, \: \forall \rho < \rho_0, \forall r>A\lambda(\rho)^{-1/\beta_1}, \\
\int_{\Rd} \left|c_d^{-1} r^{-d} n(\rho)^{-1}\mu(B(x,n(\rho)^{1/d}r))- \phi(x)\right|g(x) dx<\varepsilon.
\end{multline}
\\
Let $\varepsilon>0$ and $\rho<\rho_0$ such that $\lambda(\rho) >1$, we have:
\begin{equation}\label{cv1}
\begin{split}
\int_{A\lambda(\rho)^{-1/\beta_1}}^{+\infty}& \int_{\Rd} H_1(x,r,\rho)dxdr \\
&\leq 2 L \vert\theta \vert  c_d \int_{A\lambda(\rho)^{-1/\beta_1}}^{+\infty} \left(r^{d-\beta_1-1}\vee r^{d-\beta_2-1}\right) \int_{\Rd} \left|\frac{\mu(B(x,n(\rho)^{1/d}r))}{c_d r^{d} n(\rho)}- \phi(x)\right|g(x) dxdr \\
&\leq 2 \varepsilon L \vert\theta \vert c_d \left( \int_{A\lambda(\rho)^{-1/\beta_1}\wedge 1}^{1}  r^{d-\beta_2-1}dr+ \int_{1\vee A\lambda(\rho)^{-1/\beta_1}}^{+\infty} r^{d-\beta_1-1}dr\right)<+\infty 
\end{split}
\end{equation}
since $\beta_1-d+1 >1$.
\\
On the other hand, since $\vert \psi_G(v)\vert\leq K \vert v\vert^{\alpha}$
\begin{displaymath}
\begin{split}
\vert \psi_G\left(\theta n(\rho)^{-1}\mu(B(x,n(\rho)^{1/d}r))\right)-& \psi_G\left(\theta\phi(x)c_dr^d\right)\vert \\
&\leq K\vert \theta\vert^{\alpha}c_d^{\alpha}r^{\alpha d} \left( \frac{\vert\mu(B(x,n(\rho)^{1/d}r))\vert^{\alpha}}{(c_dr^{d}n(\rho))^{\alpha}}+\vert\phi(x)\vert^{\alpha}\right)
\end{split}
\end{displaymath}
and then
\begin{equation}
\begin{split}
H_1(x,r,\rho)&\leq K\vert \theta\vert^{\alpha}c_d^{\alpha}r^{\alpha d} \left( \frac{\vert\mu(B(x,n(\rho)^{1/d}r))\vert^{\alpha}}{(c_dr^{d}n(\rho))^{\alpha}}+\vert\phi(x)\vert^{\alpha}\right) \lambda(\rho)^{1/\beta_1}\Vert f_\rho(\cdot, \lambda(\rho)^{1/\beta_1}r)\Vert_{\infty}\\
&\leq K\vert \theta\vert^{\alpha}c_d^{\alpha}r^{\alpha d} \left( \frac{\vert\mu(B(x,n(\rho)^{1/d}r))\vert^{\alpha}}{(c_dr^{d}n(\rho))^{\alpha}}+\vert\phi(x)\vert^{\alpha}\right) \lambda(\rho)^{1+1/\beta_1}\Vert f(\cdot, \lambda(\rho)^{1/\beta_1}r)\Vert_{\infty}. 
\end{split}
\end{equation}
\\
Then by integration,  
\begin{equation}
\begin{split}
&\int_{0}^{A/\lambda(\rho)^{1/\beta_1}} \int_{\Rd} H_1(x,r,\rho) dxdr\\
 \leq &K\vert \theta\vert^{\alpha}c_d^{\alpha}\lambda(\rho)^{1+1/\beta_1} \!\! \int_0^{A/\lambda(\rho)^{1/\beta_1}} \! \!\! \!  r^{\alpha d}\Vert f(\cdot, \lambda(\rho)^{1/\beta_1}r)\Vert_{\infty}\int_{\Rd}\!\!   \left( \frac{\vert\mu(B(x,n(\rho)^{1/d}r))\vert^{\alpha}}{(c_dr^{d}n(\rho))^{\alpha}}+\vert\phi(x)\vert^{\alpha} \!\!  \right)  dxdr.
\end{split}
\end{equation}
But, 
\begin{displaymath}
\int_{\Rd}\left( \frac{\vert\mu(B(x,n(\rho)^{1/d}r))\vert^{\alpha}}{(c_dr^{d}n(\rho))^{\alpha}}+\vert\phi(x)\vert^{\alpha}\right)  dx\leq  \Vert \phi^*\Vert_{\alpha }^{\alpha} +\Vert \phi\Vert_{\alpha}^{\alpha} .
\end{displaymath}
\\
So finally,
\begin{align}\label{cv2}
&\int_{0}^{A\lambda(\rho)^{-1/\beta_1}}\int_{\Rd}H_1(x,r,\rho)dx dr \notag \\
\leq & K\vert \theta\vert^{\alpha}c_d^{\alpha} \left( \Vert \phi^*\Vert_{\alpha }^{\alpha} +\Vert \phi\Vert_{\alpha}^{\alpha}\right)\lambda(\rho)^{1+1/\beta_1}\int_0^{A\lambda(\rho)^{-1/\beta_1}}r^{\alpha d}\Vert f(\cdot, \lambda(\rho)^{1/\beta_1}r)\Vert_{\infty}dr \notag \\
\leq & K\vert \theta\vert^{\alpha}c_d^{\alpha} \left( \Vert \phi^*\Vert_{\alpha }^{\alpha} +\Vert \phi\Vert_{\alpha}^{\alpha}\right)\lambda(\rho)^{1-(\alpha d)/\beta_1} \int_0^{ A} r^{\alpha d}\Vert f(\cdot,r)\Vert_{\infty}dr \longrightarrow 0,
\end{align}
when $\rho \to 0$ (i.e. $\lambda(\rho) \to +\infty$) using Lemma \ref{lem:majf} and $\beta_1< \alpha d$.
\\
Thus \eqref{cv1} and \eqref{cv2} entail together
\begin{equation}\label{Cv1}
\lim_{\rho \to 0} \int_{0}^{+\infty}\int_{\Rd}H_1(x,r,\rho)dx dr=0.
\end{equation}
\\
We focus now on $H_2$. We treat two cases apart: $x\in B_1$ and $x\in B_1^c$. First, we consider $x\in B_1$. In that case, $$H_2(x,r,\rho)=\left| \psi_G(\theta\phi(x)c_dr^d)\right| \left|\lambda(\rho)^{1/\beta_1} f_\rho\left(x,\lambda(\rho)^{1/\beta_1} r\right)- \frac{g(x)}{r^{\beta_1+1}}\right|.$$  
\\
Performing a change of variable in \eqref{cfbis} and integrating we obtain:
\begin{multline}\label{cv3}
\forall \varepsilon >0, \: \exists A>0, \forall \rho>0, \\
 \int_{A\lambda(\rho)^{-1/\beta_1}}^{+\infty} \int_{B_1} H_2(x,r,\rho) dxdr \leq \varepsilon \int_{0}^{+\infty}\int_{B_1} \left| \psi_G(\theta\phi(x)c_dr^d)\right|  \frac{ g(x)}{r^{\beta_1+1}} dxdr.
\end{multline}
\\
Using the facts that $\vert\psi_G(u)\vert\leq K\min(\vert u\vert, \vert u\vert^{\alpha})$ (see \cite{Jcb}), $\phi \in L^q\cap L^{\alpha q}$ and Hölder's inequality we have, $\forall r \in \mathbb{R}^+$:
\begin{displaymath}
\begin{split}
\int_{B_1} &\left| \psi_G(\theta\phi(x)c_dr^d)\right|  \frac{g(x) }{r^{\beta_1+1}} dx\\
&\leq K \int_{\Rd} \min\left( \vert \theta\phi(x)\vert c_dr^d, \vert \theta \phi(x) \vert^{\alpha} c_d^{\alpha}r^{\alpha d}\right) g(x) dx r^{-\beta_1-1}\\
&\leq K \Vert g\Vert_\infty \min \left(\vert\theta \vert c_dr^d  \int_{\Rd} \vert\phi(x) \vert dx, \vert \theta \vert^{\alpha} c_d^{\alpha}r^{\alpha d}\int_{\Rd}  \vert \phi(x) \vert^{\alpha} dx\right) r^{-\beta_1-1}\\
&\leq K\Vert g\Vert_\infty \max \left(\vert\theta \vert c_d \Vert\phi \Vert_{1} , \vert \theta \vert^{\alpha} c_d^{\alpha}\Vert\phi \Vert_{\alpha }^{\alpha} \right)\left(r^{d-\beta_1-1}\wedge r^{\alpha d-\beta_1-1}\right) .\\
\end{split}
\end{displaymath}
\\
Note that $\left(r^{d-\beta_1-1}\wedge r^{\alpha d-\beta_1-1}\right)=r^{d-\beta_1-1}$ when $r\geq 1$ and  $\left(r^{d-\beta_1-1}\wedge r^{\alpha d-\beta_1-1}\right)=r^{\alpha d -\beta_1-1}$ when $r\leq 1$. Using Condition \eqref{cbeta}, $\beta_1-\alpha d+1<1<\beta_1-d+1$, $\left(r^{d-\beta_1-1}\wedge r^{\alpha d-\beta_1-1}\right)$ is integrable over $\mathbb{R}^+$ and consequently 
\begin{displaymath}
\int_{0}^{+\infty}\int_{B_1} \left| \psi_G(\theta\phi(x)c_dr^d)\right| \frac{g(x) }{r^{\beta_1+1}} dxdr<+\infty.
\end{displaymath}
\\
On the other hand, $\forall x \in B_1$, $\forall r\in \mathbb{R}^+$:
\begin{displaymath}
\begin{split}
H_2(x,r,\rho)&\leq K\vert\theta\vert^{\alpha} c_d^{\alpha}\vert\phi(x)\vert^{\alpha}r^{\alpha d}\left(\lambda(\rho)^{1+1/\beta_1}\Vert f(\cdot, \lambda(\rho)^{1/\beta_1}r)\Vert_{\infty}+\Vert g\Vert_\infty r^{-\beta_1-1}\right).
\end{split}
\end{displaymath}
 Since $\Vert \phi \Vert_{\alpha }^{\alpha}<+\infty$, we are only interested in the integral
\begin{multline*}
\int_0^{A\lambda(\rho)^{-1/\beta_1}} r^{\alpha d}\left(\lambda(\rho)^{1+1/\beta_1}\Vert f(\cdot, \lambda(\rho)^{1/\beta_1}r)\Vert_{\infty}+\Vert g\Vert_\infty r^{-\beta_1-1}\right)dr\\=\lambda(\rho)^{1-(\alpha d)/\beta_1}\int_0^A r^{\alpha d}\Vert f(\cdot, r)\Vert_{\infty}dr+\Vert g\Vert_\infty\int_0^{A\lambda(\rho)^{-1/\beta_1}}r^{\alpha d-\beta_1-1}dr.
\end{multline*}
\\
Note that, since $\alpha d-\beta_1-1>-1$, $r^{\alpha d-\beta_1-1}$ is integrable in $0$ and  $$ \int_0^{A\lambda(\rho)^{-1/\beta_1}}r^{\alpha d-\beta_1-1}dr\to 0$$ when $\rho \to 0$. Since $\int_0^A r^{\alpha d}\Vert f(\cdot, r)\Vert_{\infty}dr$ is finite (using Lemma \ref{lem:majf}), we obtain finally that:
\begin{equation}\label{cv4}
\lim_{\rho \to 0}\int_0^{A\lambda(\rho)^{-1/\beta_1}} \int_{B_1} H_2(x,r,\rho)dr= 0.
\end{equation}
 Combining \eqref{cv3} and \eqref{cv4}, we obtain that:
 \begin{equation}\label{Cv2}
\lim_{\rho\to 0} \int_{\mathbb{R}^+}\int_{B_1} H_2(x,r,\rho)dxdr=0.
 \end{equation}
\\
We deal now with the case $x\in B_1^c$. In that case $$H_2(x,r,\rho)=\left| \psi_G(\theta\phi(x)c_dr^d)\right| \left|\lambda(\rho)^{1/\beta_1} f_\rho\left(x,\lambda(\rho)^{1/\beta_1} r\right)\right|.$$
\\
Let $\rho_1$ be such that, $\forall \rho>\rho_1$, $\lambda(\rho)>1$. From Condition \eqref{cfr}, we deduce:
\begin{multline*}
\exists A>0,\: \forall \rho >\rho_1, \forall r>A\lambda(\rho)^{-1/\beta_1}, \: \forall x \in B_1^c, \\
 H_2(x,r,\rho)\leq 2\left| \psi_G(\theta\phi(x)c_dr^d)\right| \Vert g\Vert_\infty \left(r^{-\beta_1-1}\vee r^{-\beta_2-1}\right).
\end{multline*}
\\
Using $\vert\psi_G(u)\vert\leq K\min(\vert u\vert, \vert u\vert^{\alpha})$ and $\phi \in L^1(\Rd)\cap L^{\alpha }(\Rd)$ we have:
\begin{displaymath}
\begin{split}
\int_{B_1^c} &\left| \psi_G(\theta\phi(x)c_dr^d)\right| \Vert g \Vert_\infty \left(r^{-\beta_1-1}\vee r^{-\beta_2-1}\right) dx\\
&\leq K\Vert g\Vert_\infty \max \left(\vert\theta \vert c_d \Vert\phi \Vert_{1} , \vert \theta \vert^{\alpha} c_d^{\alpha}\Vert\phi \Vert_{\alpha } \right)\left(r^d\wedge r^{\alpha d}\right) \left(r^{-\beta_1-1}\vee r^{-\beta_2-1}\right).\\
\end{split}
\end{displaymath}
\\
Under Condition \eqref{cbeta} and since $\beta_2< \alpha d$,  $\left(r^d\wedge r^{\alpha d}\right) \left(r^{-\beta_1-1}\vee r^{-\beta_2-1}\right)$ is integrable over $\mathbb{R}^+$ and consequently 
\begin{displaymath}
\int_{0}^{+\infty}\int_{B_1^c} \left| \psi_G(\theta\phi(x)c_dr^d)\right| g(x)  \left(r^{-\beta_1-1}\vee r^{-\beta_2-1}\right)dxdr<+\infty.
\end{displaymath}
\\
Then by the dominated convergence theorem:
\begin{equation}\label{cv5}
\lim_{\rho \to 0} \int_{A\lambda(\rho)^{-1/\beta_1}}^{+\infty}\int_{B_1^c} H_2(x,r,\rho) dxdr =0.
\end{equation}
\\
On the other hand, for $\rho > \rho_1$ (i.e. $\lambda(\rho)>1$), $\forall r\in \mathbb{R}^+$, $\forall x\in B_1^c$:
\begin{displaymath}
\begin{split}
H_2(x,r,\rho)&\leq K\vert\theta\vert^{\alpha} c_d^{\alpha}\vert\phi(x)\vert^{\alpha}r^{\alpha d}\lambda(\rho)^{1+1/\beta_1}\Vert f(\cdot, \lambda(\rho)^{1/\beta_1}r)\Vert_{\infty}.
\end{split}
\end{displaymath}
 Since $ \Vert \phi \Vert_{\alpha }^{\alpha}<+\infty $ and
\begin{displaymath}
\int_0^{A\lambda(\rho)^{-1/\beta_1}} r^{\alpha d}\lambda(\rho)^{1+1/\beta_1}\Vert f(\cdot, \lambda(\rho)^{1/\beta_1}r)\Vert_{\infty}dr=\lambda(\rho)^{1-(\alpha d)/\beta_1}\int_0^A r^{\alpha d}\Vert f(\cdot, r)\Vert_{\infty}dr\to 0,
\end{displaymath}
when $\rho \to 0$ using Lemma \ref{lem:majf}, we obtain:

\begin{equation}\label{cv6}
\lim_{\rho \to 0}\int_0^{A\lambda(\rho)^{-1/\beta_1}} \int_{B_1^c} H_2(x,r,\rho)dr= 0.
\end{equation}
 Combining \eqref{cv5} and \eqref{cv6} we obtain that:
 \begin{equation}\label{Cv3}
\lim_{\rho\to 0} \int_{\mathbb{R}^+}\int_{B_1^c} H_2(x,r,\rho)dxdr=0,
 \end{equation}
which, combined now with \eqref{Cv2} yields that:

 \begin{equation}\label{Cv4}
\lim_{\rho\to 0} \int_{\mathbb{R}^+}\int_{\Rd} H_2(x,r,\rho)dxdr=0.
 \end{equation}
\\
From \eqref{majH}, \eqref{Cv1} and \eqref{Cv4}, we obtain:
\begin{multline}\label{cc1}
\lim_{\rho\to 0} \I \psi_G\left(\theta n(\rho)^{-1}\mu(B(x,n(\rho)^{1/d}r))\right)\frac{n(\rho)^{1/d}}{\rho}f_\rho\left(x,\frac{n(\rho)^{1/d}}{\rho}r\right)dxdr\\
=\int_{\mathbb{R}^+}\int_{B_1} \psi_G\left( \theta \phi(x) c_dr^d\right) g(x) r^{-\beta_1-1}drdx.
\end{multline}
\bigskip

We conclude the proof by proving that the right-hand side in \eqref{cc1} is the $\log$-characteristic \\ function of $\widetilde{Z}(\mu)$. Let us split the integration over $B_1$ into $\left\lbrace x\in B_1: \: \theta \phi(x) \geq 0\right\rbrace$ and \\ $\left\lbrace x\in B_1: \: \theta \phi(x) < 0\right\rbrace$ and perform a change of variable. We obtain:
\begin{displaymath}
\int_{B_1} \psi_G\left( \theta \phi(x) c_dr^d\right) g(x) r^{-\beta_1-1}drdx=D\int_{B_1} (\theta \phi(x))_+^{\gamma}g(x)dx+\bar{D}\int_{B_1}(\theta \phi(x))_-^{\gamma}g(x)dx,
\end{displaymath}
where $D=\frac{c_d^{\gamma}}{d}\int_{\mathbb{R}^+}\psi_G(r)r^{-\gamma-1}dr$ and $\bar{D}$ denotes its complex conjugate. We deduce:

\begin{displaymath}
\lim_{\rho \to 0} \log\left( \varphi_{n(\rho)^{-1}\Mr(\mu)}(\theta)\right)=-\sigma_{\phi}^{\gamma}\vert \theta \vert^{\gamma}\left(1+iB_{\phi}\epsilon(\theta) \tan\left(\frac{\pi \gamma}{2}\right)\right),
\end{displaymath}
where:
\begin{displaymath}
\sigma_{\phi}^{\gamma}=\sigma_{\gamma}^{\gamma}\int_{B_2}\vert \phi(x)\vert^{\gamma}g(x)dx,
\end{displaymath}
and, since $\gamma \in \left]1,2\right[$,
\begin{displaymath}
\begin{split}
B_{\phi}&=\frac{\int_{\mathbb{R}^+}(r-\sin(r))r^{-\gamma-1}dr}{\tan(\pi\gamma/2)\int_{\mathbb{R}^+}(1-\cos(r))r^{-\gamma-1}dr} \frac{\int_{\mathbb{R}}\epsilon(m)\vert m\vert^{\gamma}G(dm)}{\int_{\mathbb{R}}\vert m\vert^{\gamma}G(dm)} \frac{\int_{B_1}\epsilon(\phi(x))\vert\phi(x)\vert^{\gamma}g(x)dx}{\int_{B_1} \vert \phi(x)\vert^{\gamma}g(x)dx}\\
&=b_{\gamma}\frac{\int_{B_1}\epsilon(\phi(x))\vert\phi(x)\vert^{\gamma}g(x)dx}{\int_{B_1} \vert \phi(c)\vert^{\gamma}g(x)dx},
\end{split}
\end{displaymath}
because of the following identity, see $\left[\mbox{\cite{Fel}, Lemma 2}\right]$.
\begin{displaymath}
\int_{\mathbb{R}^+}(r-\sin(r))r^{-\gamma-1}dr=-\tan(\pi\gamma/2)\int_{\mathbb{R}^+}(1-\cos(r))r^{-\gamma-1}dr.
\end{displaymath}
\\
This achieves the proof of Theorem \ref{thm:sbs}. 
 
\end{proof}

\bigskip

The limiting field $\widetilde{Z}$ enjoys similar properties as $Z$ and $J$.
\begin{Prop} \quad
\begin{enumerate}
\item When $g$ is radial and $B_1$ is closed under rotation, the field $\widetilde{Z}$ is isotropic.
\item Suppose there exists $H\in \mathbb{R}$ such that, for any $a\in \mathbb{R}^+$, for any $x\in \Rd$, $g(ax)=a^Hg(x)$. If $B_1$ is closed under dilatation, $\widetilde{Z}$ is self-similar with index $(H+d-\beta_1)/\gamma$.
\end{enumerate}
\end{Prop}

\bigskip

\begin{Rq}
As in Remarks \ref{rq:lbs} and \ref{rq:is}, Theorem \ref{thm:sbs} recovers Theorem 2.16 in \cite{Jcb} when $g=1$ and $\beta$ is a constant function.
\end{Rq}

\subsection{Bridging between the large-balls scaling and the small-balls scaling}

In this section, we show that the intermediate process $J_\ell$ obtained in the intermediate scaling in Theorem \ref{thm:is} can be seen as a bridge between the process $Z$ obtained in the large-balls scaling in Theorem \ref{thm:lbs} and the process $\widetilde{Z}$ obtained in the small-balls scaling in Theorem \ref{thm:sbs}. Such bridging behavior has been evidenced for the Telecom Process in \cite{Gai}.

\begin{Thm}\label{thm:bri}
Suppose $B_1$ has a non-zero Lebesgue measure, then:
\begin{enumerate}
\item When $\ell\to +\infty$ $$\frac{J_\ell(\cdot)}{\ell^{1/\alpha}} \stackrel{\mcal}{\longrightarrow}Z(\cdot);$$ \label{thm:br1}
\item When $\ell\to 0$ $$\frac{J_\ell(\cdot)}{\ell^{d/\beta_1}} \stackrel{L^1(\mathbb{R}^d)\cap L^{\alpha }(\mathbb{R}^d)}{\longrightarrow} \widetilde{Z}(\cdot).$$
\end{enumerate}
\end{Thm}

\begin{Rq}
In dimension one, when $f(x,r)=f(r)$, Theorem \ref{thm:bri} recovers Proposition 1 and Proposition 2 in \cite{Gai}. Point \ref{thm:br1} in Theorem \ref{thm:bri} also recovers Proposition 2.13 in \cite{Jcb} when $f(x,r)=f(r)$ (in any dimension).
\end{Rq}

\begin{proof}The proof of both convergences follows the proofs of Theorem \ref{thm:lbs} and Theorem \ref{thm:sbs}.
\begin{enumerate}
\item Let $\mu\in \mcal$. We first recall the $\log$-characteristic function of $\frac{J_\ell(\mu)}{\ell^{1/\alpha}} $:
\begin{displaymath}
\log\left(\varphi_{\frac{J_\ell(\mu)}{\ell^{1/\alpha}} }(\theta)\right)=\I \psi_G\left( l^{-1/\alpha}\theta\mu(B(x,r))\right) \ell\frac{g(x)1_{B_1}(x)}{r^{\beta_1+1}}dxdr.
\end{displaymath}
When $l\to +\infty$, we have that 
\begin{displaymath}
\ell\psi_G\left( \ell^{-1/\alpha}\theta\mu(B(x,r))\right) \to -\sigma^{\alpha}\vert \theta\vert^{\alpha}\vert \mu(B(x,r))\vert^{\alpha}\left( 1-i\epsilon(\theta \mu(B(x,r)))b\tan(\pi\alpha/2)\right),
\end{displaymath}
Since this convergence is uniform both in $x$ and $r$ (see\eqref{eqg} and \eqref{pr1eq1}), it can be integrated:
\begin{multline*}
\lim_{\ell \to +\infty}\log\left(\varphi_{\frac{J_\ell(\mu)}{\ell^{1/\alpha}} }(\theta)\right)
\\= -\sigma^{\alpha}\vert \theta\vert^{\alpha}\int_{\mathbb{R}^+}\int_{B_1} \vert \mu(B(x,r))\vert^{\alpha}\left( 1-i\epsilon(\theta \mu(B(x,r)))b\tan(\frac{\pi\alpha}{2})\right)\frac{g(x)}{r^{\beta_1+1}}dxdr,
\end{multline*}
which is the $\log$-characteristic function of $Z(\mu)$.
\item Let $\mu(dx)=\phi(x)dx$ with $\phi \in L^1(\Rd)\cap L^\alpha(\Rd)$. We make the change of variable $r\mapsto \ell^{1/\beta_1}r$ in the $\log$-characteristic function of $\frac{J_\ell(\mu)}{\ell^{d/\beta_1}} $
\begin{displaymath}
\log\left(\varphi_{\frac{J_\ell(\mu)}{\ell^{d/\beta_1}}}(\theta)\right)=\I \psi_G\left( \theta \frac{\mu(B(x,\ell^{1/\beta_1}r))}{\ell^{d/\beta_1}}\right) \frac{g(x)1_{B_1}(x)}{r^{\beta_1+1}}dxdr.
\end{displaymath}
We know that
\begin{displaymath}
\psi_G\left( \theta \frac{\mu(B(x,\ell^{1/\beta_1}r))}{\ell^{d/\beta_1}}\right)\to \psi_G\left( \theta \phi(x)c_dr^d\right)
\end{displaymath}
when $\ell \to 0$ and that we can invert this limit with the integration over $\mathbb{R}^+\times \Rd$ using the same arguments we used to obtain \eqref{cc1} thus:
\begin{displaymath}
\lim_{\ell \to +\infty}\left(\varphi_{\frac{J_\ell(\mu)}{\ell^{d/\beta_1}}}(\theta)\right)= \int_{\mathbb{R}^+}\int_{B_1}\psi_G(\theta \phi(x) c_dr^d)\frac{g(x)}{r^{\beta_1+1}}dxdr,
\end{displaymath}
when $\ell\to 0$. Using the computations of Theorem \ref{thm:sbs}, we identify the right-hand side with the $\log$-characteristic function of $\tilde{Z}(\mu)$.
\end{enumerate}
\end{proof}

\bigskip
\section{Zoom-in procedure}
In the foregoing, we have dealt with the zoom-out case ($\rho\to 0$). In this section, we indicate how to adapt our previous results to the zoom-in case when $\rho\to +\infty$. In that purpose, we replace Condition \eqref{cfr} by
\begin{equation}
f_\rho(x,r)\sim_{r\to 0}\lambda(\rho) \frac{g(x)}{r^{\beta(x)+1}} \label{cfzi}.
\end{equation} 
Combined with Conditions \eqref{cf2} and \eqref{cfr1}, it requires $\beta_2<d$. The function $\beta$ is no more interpreted as a tail-index but rather as a concentration index around $0$.

The three new scaling regimes can be heuristically obtained by looking at the mean number of balls with a radius large enough, that cover the origin:
\begin{displaymath}
\begin{split}
\mathbb{E}\left[\#\left\lbrace (x,r): 0\in B(x,r), r>1\right\rbrace\right]&=\int \int_{\left\lbrace (x,r): 0\in B(x,r), r>1\right\rbrace} \rho^{-1}f_\rho(x,r/\rho)dxdr\\
& \sim_{\rho\to +\infty} \lambda(\rho) \int \int_{\left\lbrace (x,r) : 0\in B(x,r), r>1 \right\rbrace} \frac{\rho^{\beta(x)}g(x)}{r^{\beta(x)+1}} dxdr.
\end{split}
\end{displaymath}
\\
Set $B_2=\left\lbrace x\in \Rd; \beta(x)=\beta_2 \right\rbrace$. Under condition \eqref{cbeta}, for $\rho$ large enough:
\begin{multline*}
\lambda(\rho) \rho^{\beta_2} \int \int_{\left\lbrace (x,r) : 0\in B(x,r), r>1 \right\rbrace} \frac{g(x)1_{B_2}(x)}{r^{\beta(x)+1}} dxdr \\ \leq \lambda(\rho) \int \int_{\left\lbrace (x,r) : 0\in B(x,r), r>1 \right\rbrace} \frac{\rho^{\beta(x)}g(x)}{r^{\beta(x)+1}} dxdr \\ \leq \lambda(\rho) \rho^{\beta_2} \int \int_{\left\lbrace (x,r) : 0\in B(x,r), r>1 \right\rbrace} \frac{g(x)}{r^{\beta(x)+1}} dxdr.
\end{multline*}
Since $\int \int_{\left\lbrace (x,r) : 0\in B(x,r), r>1 \right\rbrace} \frac{g(x)}{r^{\beta(x)+1}} dxdr<+\infty$, if the set $B_2$ is non-negligible,  three regimes appear when $\rho \to +\infty$:
\begin{itemize}
\item large-balls scaling: $\lambda(\rho)\rho^{\beta_2}\to 0$,
\item intermediate scaling: $\lambda(\rho)\rho^{\beta_2} \to \ell \in (0,+\infty)$,
\item small-balls scaling: $\lambda(\rho)\rho^{\beta_2 } \to + \infty$.
\end{itemize}
We observe the opposite phenomenon to that observed in the zoom-out case. When the mean radius increases, the biggest radii are those with the biggest concentration index around $0$, i.e. $\beta_2$. The normalization will compensate the increasing of the balls with the biggest radius and the other balls will get negligible. Note that $\lambda(\rho)$ is no longer going to $+\infty$ when $\rho \to +\infty$ except (possibly) in the small-balls scaling. As it was the case in \cite{Bek} and \cite{Jcb}, no limit process are obtained in the large-balls scaling with the zoom-in procedure. Indeed to compensate the increasing of the radii, we need to reduce the density of balls so much that nothing remains at the limit. In the small-balls  scaling and in the intermediate scaling, we obtain the following analogues of Theorems \ref{thm:lbs} and \ref{thm:is}:
\begin{Thm}
Suppose that $\lambda(\rho)\rho^{\beta_2} \to + \infty$ when $\rho \to +\infty$. Let $n(\rho)=\lambda(\rho)^{1/\alpha} \rho^{\beta_2/\alpha}$ and suppose that $B_2=\left\lbrace x\in \Rd; \: \beta(x)=\beta_2 \right\rbrace$ has a non-zero Lebesgue measure, then we have:
\begin{displaymath}
n(\rho)^{-1}\Mr(\cdot) \stackrel{\mcal}{\longrightarrow} Z_2(\cdot),
\end{displaymath}
where $Z_2(\mu)=\I \mu(B(x,r))M'_{\alpha}(dx,dr)$ is a stable integral with respect to the $\alpha$-stable measure $M'_{\alpha}$ with control measure $\sigma^{\alpha}g(x)1_{B_2}(x)r^{-\beta_2-1}dxdr$ and constant skewness function $b$ given by $G$.
\end{Thm}

\bigskip

\begin{Thm}
Suppose $\lambda(\rho)\rho^{\beta_2} \to \ell\in \left]0, +\infty \right[$ when $\rho \to +\infty$ and suppose again  $B_2$ has a non-zero Lebesgue measure, then we have:
\begin{displaymath}
\Mr(\cdot) \stackrel{\mcal}{\longrightarrow} J_\ell(\cdot), 
\end{displaymath}
where $J_\ell(\mu)=\II m \mu(B(x,r)) \tilde{\Pi}'_\ell(dx,dr,dm)$ and $\tilde{\Pi}'_\ell$ is a compensated Poisson random measure on $\mathbb{R}^d\times \mathbb{R}^+\times \mathbb{R}$ with intensity $\ell g(x)1_{B_2}(x)dx r^{-\beta_2-1}drG(dm)$.
\end{Thm}

\bigskip

Since the proofs are adaptations of the proofs of section \ref{sec:thm}, we skip them.

\section*{Appendix}
We prove Proposition \ref{propm}.

\begin{proof}     
In what follows, we note $a\vee b=\max(a,b)$.

\begin{enumerate}
\item Let $\mu_1$ and $\mu_2\in \mcal$ and $k\in \mathbb{R}$. Let $s_1,t_1, C_1$ and $s_2,t_2,C_2$ associated respectively to $\mu_1$ and $\mu_2$ according to the definition of $\mcal$. Set $s=s_1\vee s_2<\beta_1\leq \beta_2<t=t_1\wedge t_2$ and $C=C_1\vee C_2$ then, since $\alpha>1$, we have
\begin{displaymath}
\begin{split}
\int_{\Rd}\vert (\mu_1+k\mu_2)&(B(x,r))\vert^{\alpha}dx \\
\leq &\left(\int_{\Rd}\vert \mu_1(B(x,r))\vert^{\alpha}dx+k^{\alpha}\int_{\Rd}\vert\mu_2(B(x,r))\vert^{\alpha}dx\right) 2^{\alpha-1}\\
\leq &\left( C_1(r^{s_1}\wedge r^{t_1})+k^{\alpha}C_2(r^{s_2}\wedge r^{t_2})\right) 2^{\alpha-1}\\
\leq &(1+k^{\alpha})C2^{\alpha-1}(r^s\wedge r^t).
\end{split}
\end{displaymath} 
Thus $\mu_1+k\mu_2\in \mcal$ and $\mcal$ is indeed a linear space.

\bigskip

Let $\mu\in \mcal$ and $s,t, C$ associated to $\mu$, then:
\begin{displaymath}
\begin{split}
\I \! \! \! \vert \mu(B(x,r)) \vert^{\alpha}g(x)r^{-\beta(x)-1}dxdr \leq &\Vert g\Vert_{\infty}\! \! \I  \! \! \! \vert \mu(B(x,r)) \vert^{\alpha}dx\left( r^{-\beta_1-1}\vee r^{-\beta_2-1}\right) dr\\
\leq &C\Vert g\Vert_{\infty}\int_{\mathbb{R}^+} \left( r^s\wedge r^t\right)\left( r^{-\beta_1-1}\vee r^{-\beta_2-1}\right) dr.
\end{split}
\end{displaymath}
Note that $(r^t\wedge r^s)\left( r^{-\beta_1-1}\vee r^{-\beta_2-1}\right)=r^{t-\beta_2-1}$ when $r\leq1$ and \\ $(r^t\wedge r^s)\left( r^{-\beta_1-1}\vee r^{-\beta_2-1}\right)= r^{s-\beta_1-1}$ when $r>1$. Since $\beta_2-t+1<1$ and $\beta_1-s+1>1$, $(r^t\wedge r^s)\left( r^{-\beta_1-1}\vee r^{-\beta_2-1}\right)$ is integrable over $\mathbb{R}^+$ which proves \eqref{maj}.
\item Let $\mu\in \ma$ and $\alpha \leq \alpha'$.
\begin{displaymath}
\begin{split}
\int_{\Rd} \vert \mu(B(x,r)) \vert^{\alpha '}dx &\leq \int_{\Rd} \vert \mu(B(x,r)) \vert^{\alpha '-\alpha}\vert \mu(B(x,r))\vert^{\alpha}dx\\
&\leq \vert\mu\vert(\Rd)^{\alpha '-\alpha}\int_{\Rd} \vert \mu(B(x,r)) \vert^{\alpha}dx\\
&\leq C'(r^s\wedge r^t),
\end{split}
\end{displaymath} 
where $C'=\vert\mu\vert(\Rd)^{\alpha'-\alpha}C$, which proves $\mu\in \maa$.
\item Let $\mu\in \ma$. Then there exist a constant $C$ and $s<\beta_1\leq \beta_2<t$ satisfying the definition. Necessarily, $s<\beta_1'\leq \beta_2'<t$ which entails $\mu \in \mathcal{M}_{\alpha, \beta_1', \beta_2'}$.
\item Let $\mu \in \mcal$ and suppose that $\mu$ has an atom in $a\in \Rd$. Let $\varepsilon >0$ be such that $$ \big\vert \vert\mu\vert\left(B(a,\varepsilon)\right)-\vert\mu\vert(a)\big\vert\leq \vert \mu\vert (a)/2.$$ Then for every $r<\varepsilon/2$ and $x\in B(a,r)$, $\vert\mu (B(x,r))\vert\geq \vert\mu\vert(a)/2$. Integrating on $x\in B(a,r)$ we get:
\begin{displaymath}
\begin{split}
\int_{\Rd}\vert \mu(B(x,r))\vert^{\alpha}dx & \geq \int_{B(a,r)}\vert \mu(B(x,r))\vert^{\alpha}dx\\
&\geq \left( \vert\mu\vert(a)/2\right)^{\alpha}\int_{B(a,r)}dx \\
&\geq \left( \vert\mu\vert(a)/2\right)^{\alpha} c_dr^{d}. 
\end{split}
\end{displaymath}
This is in contradiction with $\mu \in \mcal$ which requires $$\int_{\Rd}\vert \mu(B(x,r))\vert^{\alpha}dx\leq Cr^t $$ for $t>\beta_1>d$ when $r$ is small enough.
\end{enumerate}
\end{proof}
\noindent We prove Lemma \ref{lem:majf}.

\begin{proof}
Let $A>0$ and $t>d$, we have
\begin{displaymath}
\begin{split}
\int_0^A r^t \Vert f(\cdot,r) \Vert_\infty dr &\leq \int_0^{1\wedge A} r^t \Vert f(\cdot, r)\Vert_\infty dr + \int_{1\wedge A}^A r^t\Vert f(\cdot,r)\Vert_\infty dr \\
& \leq \int_0^{1\wedge A} r^d \Vert f(\cdot, r)\Vert_\infty dr + \int_{1\wedge A}^A r^t\Vert f(\cdot,r)\Vert_\infty dr.
\end{split}
\end{displaymath}
The first integral above is finite because of Condition \eqref{cf2}. The integral is finite second because, using condition \eqref{cf1}, $r\mapsto r^t\Vert f(\cdot, r)\Vert_\infty$ is continuous on the bounded interval $\left[1\wedge A,A\right]$. Thus we obtain
\begin{displaymath}
\int_0^A r^t \Vert f(\cdot,r) \Vert_\infty dr<+\infty.
\end{displaymath}  
\end{proof}

\section*{Acknowledgment}

The author thanks his Ph.D advisor Jean-Christophe Breton for his precious help during the redaction of this paper and the Lebesgue center of mathematics ("Investissements d'avenir" program ---
ANR-11-LABX-0020-01) for its financial support.

\newpage

\bibliographystyle{plain}
\bibliography{bib.bib}
\end{document}